\documentclass[11pt,oneside]{article}

\usepackage{latexsym}
\usepackage{shortvrb}
\usepackage{graphicx}
\usepackage{fancyhdr}
\usepackage{amsfonts}
\usepackage{amsmath}
\usepackage{amssymb}
\usepackage{mathrsfs}
\usepackage{makeidx}
\usepackage[all]{xy}
\usepackage{tikz}
\usepackage{amsthm}

\usepackage[T1]{fontenc}
\usepackage{amsfonts}
\usepackage{amsmath}

\newtheorem{thm}{Theorem}[section]
\newtheorem{lem}[thm]{Lemma}
\newtheorem{prop}[thm]{Proposition}

\newtheorem{cor}[thm]{Corollary}

\theoremstyle{definition}
\newtheorem{definition}[thm]{Definition}

\newtheorem{conj}[thm]{Conjecture}

\newcommand{\blackged}{\hfill$\blacksquare$}
\newcommand{\whiteged}{\hfill$\square$}
\newcommand{\Hom}{\mbox{\rm Hom}}
\newcounter{proofcount}

\renewenvironment{proof}[1][\proofname.]{\par
  \ifnum \theproofcount>0 \pushQED{\whiteged} \else \pushQED{\blackged} \fi%
  \refstepcounter{proofcount}
  \normalfont %\topsep6\p@\@plus6\p@\relax
  \trivlist
  \item[\hskip\labelsep
        \itshape
    {\bf\em #1}]\ignorespaces
}{%
  \addtocounter{proofcount}{-1}
  \popQED\endtrivlist
}

\begin{document}

\begin{center}
\textbf{\large{Combinatorial Dimensions: Indecomposability on Certain Local Finite Dimensional Trivial Extension Algebras}}
\end{center}

\begin{center}
\large{Juan Orendain}
\end{center}

\noindent \textit{Abstract}: We study problems related to indecomposability of modules over certain local finite dimensional trivial extension algebras. We do this by purely combinatorial methods. We introduce the concepts of graph of cyclic modules, of combinatorial dimension, and of fundamental combinatorial dimension of a module. We use these concepts to establish, under favorable conditions, criteria for the indecomposability of a module. We present categorified versions of these constructions and we use this categorical framework to establish criteria for the indecomposability of modules of infinite rank.

\tableofcontents

\section{Introduction}

\noindent Techniques for establishing indecomposability of modules range from different areas of ring and module theory such as, the study of rings of endomorphisms [4], the study of divisibility conditions on base rings [13], [14], and more classically, the study of subsets of lattices of submodules. In this paper we present conditions for indecomposability established by purely combinatorial methods. More precisely, we reduce, in favorable cases, the condition of a module being indecomposable to the study of connectiviy properties of certain graphs. These techniques are especially suited for the study of indecomposability over certain finite dimensional trivial extension algebras. We thus apply our results to the problem of existence of indecomposable modules satisfying given conditions over finite dimensional trivial extension algebras. We now sketch the contents of this paper. 

In section 2 we define and study the concept of graph of cyclic modules of a module with respect to an ideal. We use this concept to define the concepts of combinatorial and fundamental combinatorial dimensions of a module. We prove that both, the combinatorial dimension and the fundamental combinatorial dimension of a module, when defined and finite, bound the Krull-Schmidt length of the module. We derive from this, two criteria for the indecomposability of a module. In section 3 we apply results obtained in section 2 to the problem of existence of indecomposable modules with given Goldie dimension over certain local finite dimensional trivial extension algebras. In section 4 we generalize the concepts and constructions introduced in section 3 and we propose two conjectures implying the existence of indecomposable modules of arbitrarily large Goldie dimension over trivial extension algebras studied in section 3. In section 5 we introduce a construction which we regard as a categorification of the concept of graph of cyclic modules. More precisely, we construct a family of concrete functors between given concrete categories, that specialize, under certain conditions, to the concept of graph of cyclic modules. We study functorial properties of these functors, together with categorical properties of their source and target categories. In section 6 we apply the results obtained in section 5 to the study of the problem of existence of modules of infinite Goldie dimension over trivial extension algebras studied in section 3. More generally, we introduce the concept of a rank function and we establish two criteria for the existence of indecomposable modules with infinite rank with respect to a given rank function. Finally, in section 7 we present approximation methods for the computation of graphs of cyclic modules related to methods for computation performed in previous sections.

We assume, unless otherwise stated, that an associative unital ring $R$ is chosen. Thus, unless otherwise stated, the word \textit{module} will always mean a unital left $R$-module, and the word \textit{morphism} will mean a morphism in the category of left $R$-modules, $R$-Mod. We use the usual conventions for concepts related to category theory and module theory. For related terminology in category theory see [1] or [18], and for related terminology in module theory see [20] or [10]. We use the usual conventions for concepts related to graph theory, except for the following exceptions: We will assume all graphs to have exactly one loop on each vertex, we will not assume graphs to have a finite number of vertices, and we will use, throughout, the symbol $\left|\Gamma\right|$, to denote the vertex set of a given graph $\Gamma$. We refer the reader to [5] for all concepts related to graph theory.

\section*{Acknowledgements}

\noindent  The author would like thank professors Jos\'e R\'\i{}os, Hugo Rinc\'on, and Sergio L\'opez-Permouth for their support and kind words, and the annonymous referee for very helpful sugestions. Finally, the author would like to dedicate this paper to the memory of the late professor Francisco Raggi, without whose guidance, encouragement, and friendship, this work would have not been possible.

\section{Combinatorial Dimensions}

\noindent Recall [20] that the Krull-Schmidt length, $KS\ell(M)$, of a module $M$, is defined as the supremum of cardinalities of direct sum decompositions of $M$. Thus, the Krull-Shmidt length, $KS\ell(M)$, of a module $M$, is equal to $1$ if and only if $M$ is indecomposable. Observe also that if the Krull-Schmidt length, $KS\ell(M)$, of a module $M$, is finite, then $M$ admits finite indecomposable direct sum decompositions. In this first section we introduce the concepts of combinatorial dimension and fundamental combinatorial dimension of a module. We prove that both the combinatorial dimension and the fundamental combinatorial dimension of a module $M$, when defined and finite, are upper bounds for the Krull-Schmidt length, $KS\ell(M)$, of $M$, thus providing criteria for the indecomposability of $M$. We begin with the following definition. 

\begin{definition}
Let $I$ be an ideal, $M$ a module, and $\Sigma$ a subset of $M\setminus\left\{0\right\}$. We define the $I$-graph of cyclic modules of $\Sigma$ in $M$, $\Gamma_I(M,\Sigma)$, as follows:
\begin{enumerate}
\item The set of vertices $\left|\Gamma_I(M,\Sigma)\right|$ of $\Gamma_I(M,\Sigma)$ will be the set of all cyclic submodules of $M$, generated by elements of $\Sigma$.
\item If $a,b\in \Sigma$, then $Ra$ and $Rb$ will be adjacent in $\Gamma_I(M,\Sigma)$ if 

\[Ia\cap Ib\neq\left\{0\right\}\]

\end{enumerate}
\noindent Moreover, if $\Sigma=M\setminus I M$, we will write $\Gamma_I(M)$ for $\Gamma_I(M,\Sigma)$, and in this case we will call $\Gamma_I(M)$ the $I$-graph of cyclic modules of $M$.
\end{definition}

\noindent Thus, in the case where $I=\left\{0\right\}$, the $I$-graph of cyclic modules of $\Sigma$ in $M$, $\Gamma_I(M,\Sigma)$, is equal to the discrete graph generated by the set of all cyclic submodules of $M$, generated by elements of $\Sigma$, that is, $\Gamma_I(M,\Sigma)$ has the set of all cyclic submodules of $M$, generated by elements of $\Sigma$, as set of vertices, $\left|\Gamma_I(M,\Sigma)\right|$, and its only edges are a loop on each of its vertices. If $R$ is a field, then $\Gamma_R(M,\Sigma)$ is again the discrete graph generated by the set of all one dimensional subspaces of $M$, generated by elements of $\Sigma$. If now $I=R$ and $M$ is uniform, then $\Gamma_I(M,\Sigma)$ is the complete graph generated by all cyclic submodules generated by elements of $\Sigma$. In particular, graphs

\[\Gamma_{\left\{0\right\}}(\mathbb{Z},\mathbb{Z}\setminus\left\{0\right\}), \ \Gamma_{\mathbb{F}_2}(\mathbb{F}_2^2,\mathbb{F}_2^2\setminus\left\{0\right\}), \ \Gamma_{\mathbb{Z}}(\mathbb{Z},\mathbb{Z}\setminus\left\{0\right\})\]

\noindent are: The discrete graph generated by a countable set of vertices, the discrete graph generated by three vertices, and the complete graph generated by a countable set of vertices respectively.

Given an ideal $I$ and a module $M$ we will denote by $ann_IM$ the subset of $M$ annihilated by $I$, that is, $ann_IM$ will denote the set of all $a\in M$ such that $Ia=\left\{0\right\}$. Further, we will denote by $ann_I^*M$ the subset of $M$ pointwise annihilated by $I$, that is $ann_I^*M$ will denote the set of all $a\in M$ such that there exists $r\in I\setminus\left\{0\right\}$ such that $ra=0$. It is easily seen that in general, $ann_IM\subseteq ann_I^*M$ whenever $I\neq\left\{0\right\}$, and that $I^2M=\left\{0\right\}$ if and only if $I M\subseteq ann_IM$. We will say that the ideal $I$ is a decomposition ideal for $M$ if the opposites of these contentions hold, that is, we will say that $I$ is a decomposition ideal of $M$ if the equalities 
\[I M=ann_IM=ann_I^*M\]
hold. We will denote by $\mathfrak{D}(I)$ the domain of decomposition of the ideal $I$, that is, $\mathfrak{D}(I)$ will denote the class of all modules $M$ such that $I$ is a decomposition ideal of $M$. Further, we will denote by $\mathfrak{D}^{-1}(M)$ the codomain of decomposition of $M$, that is, $\mathfrak{D}^{-1}(M)$ will denote the set of all ideals $I$ such that $M\in\mathfrak{D}(I)$. Given an ideal $I\in\mathfrak{D}^{-1}(M)$, we define the $I$-combinatorial dimension of $M$, ($cdim_IM$), as the number of connected components of the $I$-graph of cyclic modules, $\Gamma_I(M)$, of $M$. If $\mathfrak{D}^{-1}(M)\neq\emptyset$, we define then, the combinatorial dimension ($cdimM$) of $M$ as the minimum of $I$-combinatorial dimensions of $M$, with $I\in\mathfrak{D}^{-1}(M)$. The following theorem is the main result of this section.

\begin{thm}
Let $M$ be a module. If $cdimM$ exists and is finite, then 
\[KS\ell(M)\leq cdimM\]
\end{thm} 

\noindent As a direct consequence of (2.2) we have the following corollary.

\begin{cor}
Let $M$ be a module
\begin{enumerate}
\item If $cdimM$ exists and is finite, then $M$ admits finite indecomposable direct sum decompositions.
\item If $cdimM=1$ then $M$ is indecomposable.
\end{enumerate}
\end{cor}

\noindent We now reduce the proof of (2.2) to the proof of a series of lemmas.

\noindent Given an ideal $I$, a module $M$ and a submodule $N$ of $M$, we will say that $N$ is $I$-pure in $M$ ($N\leq_IM$) if the equality
\[IN=IM\cap N\]
holds. If further, $N$ is $I$-pure in $M$ for every ideal $I$, we will say that $N$ is pure in $M$. Direct summands are examples of pure submodules.

\begin{lem}
Let $I$ be an ideal. Let $M$ and $N$ be modules. If $N\leq_IM$, then
\begin{enumerate}
\item If $M\in\mathfrak{D}(I)$, then $N\in\mathfrak{D}(I)$.
\item $\Gamma_I(N)$ is a complete subgraph of $\Gamma_I(M)$.
\end{enumerate}
\end{lem}

\begin{proof}
Suppose that $M\in\mathfrak{D}(I)$. From the fact that $IM\subseteq ann_IM$ it follows that $IN\subseteq ann_IN$. Now, $ann^*_IN=ann^*_IM\cap N$, and since $M\in\mathfrak{D}(I)$, this is equal to $IM\cap N$. This in turn is equal to $IN$ since $N\leq_IM$. It follows that $IN=ann_IN=ann^*_IN$. We conclude that $N\in\mathfrak{D}(I)$. Now, from the fact that $N\leq_IM$ it easily follows that $N\setminus IN\subseteq M\setminus IM$. It follows that $\left|\Gamma_I(N)\right|\subseteq \left|\Gamma_I(M)\right|$. It is immediate that the relation of adjacency in $\Gamma_I(N)$ is the restriction of the relation of adjacency in $\Gamma_I(M)$. We conclude that $\Gamma_I(N)$ is a complete subgraph of $\Gamma_I(M)$.
\end{proof}

\noindent It follows in particular, from (2.4) that the domain of decomposition, $\mathfrak{D}(I)$, of any ideal $I$, is closed under the operation of taking direct summands, and that for any ideal $I$ and module $M$, the $I$-graph of cyclic modules, $\Gamma_I(N)$, of any direct summand $N$ of $M$, is a complete subgraph of the $I$-graph of cyclic modules, $\Gamma_I(M)$, of $M$.

\begin{lem}
Let $M$ be a module. Let $M=A\oplus B$ be a direct sum decomposition of $M$. Let $I\in\mathfrak{D}^{-1}(M)$ and $Rx$ be a vertex of $\Gamma_I(M)$. If there exists a vertex, $Ra$, of $\Gamma_I(A)$ such that $Rx$ and $Ra$ are adjacent in $\Gamma_I(M)$, then $x\in A\oplus IB$. 
\end{lem}

\begin{proof}
Suppose $Rx$ and $Ra$ are adjacent in $\Gamma_I(M)$. Then there exist $r,r'\in I\setminus\left\{0\right\}$ such that $rx=r'a$. If we let $x=\alpha+\beta$ with $\alpha\in A$ and $\beta\in B$, then $rx=r\alpha +r\beta=r'a$. It follows that $r\beta=r'a-r\alpha$, from which it follows that $r\beta\in A\cap B=\left\{0\right\}$. Thus $\beta\in ann^*_IM$. It follows that $\beta\in IB$. This concludes the proof.
\end{proof}

\begin{lem}
Let $M$ be a module. Let $M=A\oplus B$ be a direct sum decomposition of $M$. Let $I\in\mathfrak{D}^{-1}(M)$. If $Rx_1,...,Rx_n$ is a path in $\Gamma_I(M)$ such that $x_1\in A$, then there exists a path $Ra_1,...,Ra_n$ in $\Gamma_I(A)$ such that $x_1=a_1$ and such that $Ix_i=Ia_i$ for all $1\leq i\leq n$.
\end{lem}

\begin{proof}
We do induction on $n$. The case in which $n=1$ is trivial. Suppose now that the result is true for $n$. Let $Rx_1,...,Rx_{n+1}$ be a path in $\Gamma_I(M)$ such that $x_1\in A$. By induction hypothesis there exists a path $Ra_1,...,Ra_n$ in $\Gamma_I(A)$ such that $a_1=x_1$ and such that $Ix_i=Ia_i$ for all $1\leq i\leq n$. From the fact that $Ia_n=Ix_n$ together with the fact that $Rx_n$ and $Rx_{n+1}$ are adjacent in $\Gamma_I(M)$ it follows that $Ra_n$ and $Rx_{n+1}$ are adjacent in $\Gamma_I(M)$. It follows, from (2.5), that $x_{n+1}\in A\oplus IB$. Thus, if $x_{n+1}=\alpha+\beta$ with $\alpha\in A$ and $\beta\in B$, then $\beta\in IB=ann_I^*B$, from which it follows that $Ix_{n+1}=I(\alpha+\beta)=I\alpha$. From this it follows that if we make $a_{n+1}=\alpha$, then $Ia_{n+1}=Ix_{n+1}$ and $Ra_n$ and $Ra_{n+1}$ are adjacent in $\Gamma_I(M)$. This concludes the proof.  
\end{proof}

\begin{cor}
Let $M$ be a module. Let $M=A\oplus B$ be a direct sum decomposition of $M$. Let $I\in\mathfrak{D}^{-1}(M)$ and $Rx$ be a vertex of $\Gamma_I(M)$. If there exists a vertex $Ra$ of $\Gamma_I(A)$ such that $Rx$ and $Ra$ are in the same connected component of $\Gamma_I(M)$, then $x\in A\oplus IB$.
\end{cor}

\begin{proof}
Suppose $Rx$ and $Ra$ are in the same connected component of $\Gamma_I(M)$. Then there exists a path $Rx_1,...,Rx_n$ such that $x_1=a$ and such that $x_n=x$. From (2.6) it follows that there exists a path $Ra_1,...,Ra_n$ in $\Gamma_I(A)$ such that $a_1=a$ and such that $Ix_i=Ia_i$ for all $1\leq i\leq n$. In particular $Ix_n=Ia_n$. From this and from (2.5) it follows that $x\in A\oplus IB$. This concludes the proof. 
\end{proof}

\begin{cor}
Let $M$ be a module. Let $M=A\oplus B$ be a direct sum decomposition of $M$. Let $I\in\mathfrak{D}^{-1}(M)$. If $Ra,Rb$ are two vertices of $\Gamma_I(A)$ then $Ra$ and $Rb$ are in the same connected component of $\Gamma_I(M)$ if and only of $Ra$ and $Rb$ are in the same connected component of $\Gamma_I(A)$.
\end{cor}

\begin{proof}
Suppose $Ra$ and $Rb$ are in the same connected component of $\Gamma_I(M)$. Then there exists a path $Rx_1,...,Rx_n$ in $\Gamma_I(M)$ such that $x_1=a$ and such that $x_n=b$. From (2.6) it follows that there exists a path $Ra_1,...,Ra_n$ in $\Gamma_I(A)$ such that $a_1=x_1=a$ and such that $Ia_i=Ix_i$ for every $1\leq i\leq n$. In particular $Ia_n=Ib$. Thus the path $Ra_1,...,Ra_{n-1},Rb$ is a path in $\Gamma_I(A)$ between $Ra$ and $Rb$. We conclude that $Ra$ and $Rb$ are in the same connected component of $\Gamma_I(A)$. This concludes the proof. 
\end{proof}

\begin{lem}
Let $M$ be a module. Let $I\in\mathfrak{D}^{-1}(M)$. If $\Gamma_I(M)$ is connected, then $M$ is indecomposable. 
\end{lem}

\begin{proof}
Suppose that $I\in\mathfrak{D}^{-1}(M)$ and that $\Gamma_I(M)$ is connected. Let $M=A\oplus B$ be a direct sum decomposition of $M$. Suppose $A\neq \left\{0\right\}$. Then, since $IA$ is superfluous in $A$ and thus it is superfluous in $M$, it follows that $\left|\Gamma_I(A)\right|\neq\emptyset$. Let $x\in M\setminus IM$. From the fact that $\Gamma_I(M)$ is connected and from the fact that $\left|\Gamma_I(A)\right|\neq\emptyset$ it follows  , from (2.7), that $x\in A\oplus IB$. It follows that $B=IB$. From this and from the fact that $I^2B=\left\{0\right\}$ we conclude that $B=\left\{0\right\}$. This concludes the proof. 
\end{proof}

\noindent\textbf{\textit{Proof of 2.2}} Suppose $cdimM$ exists and is finite. Let $I\in \mathfrak{D}^{-1}(M)$ such that $cdim_IM$ is finite. We do induction on $cdim_IM$. The base of the indction is (2.9). Suppose the result is true for $n$. Suppose now that $cdim_IM=n+1$. Let $M=\bigoplus_{\alpha\in A}M_\alpha$ be a direct sum decomposition of $M$ such that $M_\alpha\neq\left\{0\right\}$ for every $\alpha\in A$. Let $\alpha_0\in A$. By (2.8) every connected component of both $\Gamma_I(A)$ and of $\Gamma_I(\bigoplus_{\alpha\in A\setminus\left\{\alpha_0\right\}}M_\alpha)$ is contained in a single connected component of $\Gamma_I(M)$. Since $\Gamma_I(M_{\alpha_0})$ has at least one component it follows that 
\[cdim_I\bigoplus_{\alpha\in A\setminus\left\{\alpha_0\right\}}M_\alpha\leq n\]
From the induction hypothesis it follows that $\left|A\setminus\left\{\alpha_0\right\}\right|\leq n$, that is $\left|A\right|\leq n+1$. This concludes the proof. $\blacksquare$

\

\noindent Thus the $I$-combinatorial dimension, $cdim_IM$, of a module $M$ with respect to an ideal $I\in\mathfrak{D}^{-1}(M)$, when finite, is an upper bound for the Krull-Schmidt length, $KS\ell(M)$, of $M$. We now improve this bound in the case where the $I$-graph of cyclic modules, $\Gamma_I(M)$, of $M$ satisfies certain conditions. Given an ideal $I$, a module $M$, and $x\in M\setminus IM$, we will denote the connected component of $Rx$ in $\Gamma_I(M)$ by $C_x^I$. We will say that a subset $\Sigma$ of $M\setminus IM$ is fundamental with respect to $I$ if $\Sigma$ generates $M$ and if the inequality

\[\sum_{\left|\Gamma_I(M)\right|\setminus C_x^I} Ry\neq M\]
  
\noindent holds for every $x\in \Sigma$. We will denote by $\mathfrak{F}(I)$ the fundamental domain of decomposition of $I$, that is, $\mathfrak{F}(I)$ will denote the class of all modules $M\in\mathfrak{D}(I)$ such that $M$ has fundamental subsets with respect to $I$. Further, we will denote by $\mathfrak{F}^{-1}(M)$ the fundamental codomain of decomposition of $M$, that is, $\mathfrak{F}^{-1}(M)$ will denote the set of all ideals $I$ such that $M\in\mathfrak{F}(I)$. The following lemma says that the set of connected components generated by cyclic modules generated by elements of a fundamental subset $\Sigma$ of a module $M$ wih respect to an ideal $I$, does not depend on $\Sigma$.

\begin{lem}
Let $I$ be an ideal. Let $M$ be a module. If $\Sigma,\Sigma'$ are fundamental subsets of $M$, with respect to $I$, then
\[\left\{C_x^I:x\in \Sigma\right\}=\left\{C_y^I:y\in \Sigma'\right\}\]
\end{lem}

\begin{proof}
Suppose $\Sigma,\Sigma'$ are fundamental subsets of $M$, with respect to $I$. Suppose there exists $x\in \Sigma$ such that there does not exist $y\in \Sigma'$ such that $C_x^I=C_y^I$. The inequality 
\[\sum_{y\in \Sigma'}Ry\leq \sum_{\Gamma_I(M)\setminus C_x^I}Ry\]
follows. The left hand side of this inequality is equal to $M$ while the right hand side is a proper submodule of $M$, a contradiction. This completes the proof.
\end{proof}

\noindent Thus, given an ideal $I$ and a module $M$, we define the $I$-fundamental combinatorial dimension of $M$ ($fcdim_IM$) as the cardinality of the set of connected components of $\Gamma_I(M)$ generated by cyclic modules generated by any fundamental subset of $M$ with respect to $I$. If $\mathfrak{F}^{-1}(M)\neq \emptyset$, we define the fundamental combinatorial dimension of $M$ ($fcdimM$) as the minimum of $I$-fundamental combinatorial dimensions of $M$, with $I\in\mathfrak{F}^{-1}(M)$.

\begin{thm}
Let $I$ be an ideal. Let $M$ be a module. If $fcdim_IM$ exists and is finite then
\[KS\ell(M)\leq fcdim_IM\leq cdim_IM\]
\end{thm} 

\begin{proof}
Let $\Sigma$ be a fundamental subset of $M$ with respect to $I$. Let $C_1,...,C_n$ be the connected components of $\Gamma_I(M)$ defined by elements of $\Sigma$. Let $M=\bigoplus_{\alpha\in A}M_\alpha$ be a direct sum decomposition of $M$ such that $M_\alpha\neq\left\{0\right\}$ for every $\alpha\in A$. It is easily seen that the sum of all vertices of all graphs $\Gamma_I(M_\alpha)$, $\alpha\in A$, is equal to $M$. From this it follows that there exist $\alpha_1,...,\alpha_n\in A$, not necessarely different, such that 
\[C_i\cap \left|\Gamma_I(M_{\alpha_i})\right|\neq\emptyset\]
for every $1\leq i\leq n$. From this and from (2.8) it follows that $C_i\subseteq \left|\Gamma_I(M_{\alpha_i})\right|$ for every $1\leq i\leq n$. Thus 

\[\sum_{x\in \Sigma}Rx\leq \bigoplus_{i=1}^nM_{\alpha_i}\]

\noindent From the fact that the left hand side of this identity is equal to $M$ it follows that $\left|A\right|\leq n$. We conclude that $KS\ell(M)\leq fcdim_IM$. The right hand side of this last inequality is clearly less than or equal to $cdim_IM$. This concludes the proof. 
\end{proof}

\begin{cor}
Let $M$ be a module. If $fcdimM$ exists and is finite then
\[KS\ell(M)\leq fcdimM\]
\end{cor}

\begin{cor}
Let $M$ be a module
\begin{enumerate}
\item If $fcdimM$ exists and is finite, then $M$ admits finite indecomposable direct sum decompositions.
\item If $fcdimM=1$, then $M$ is indecomposable.
\end{enumerate}
\end{cor}

\section{Indecomposability on trivial extensions}

\noindent Given a field $k$ and a $k$-vector space $V$, the trivial extension of $k$ by $V$, $k\ltimes V$, is defined as follows: As a $k$-vector space, $k\ltimes V$ will be equal to $k\oplus V$. Given $(a,v),(b,w)\in k\ltimes V$, we define $(a,v)(b,w)$ as $(ab,bv+aw)$. It is easily seen that with this structure $k\ltimes V$ is a unital commutative $k$-algebra with $(1_k,0)$ as $1$. This is a special case of a much more general construction (see [12], [16]). In this section we apply results obtained in section 2 to the study of the problem of existence of indecomposable modules of given Goldie dimensions on algebras of the form $k\ltimes V$, where $V$ is a finite dimensional $k$-vector space.

\noindent Observe that given a field $k$ and a $k$-vector space $V$, the dimension of $k\ltimes V$ as a $k$-vector space is $dim_kV+1$, and that the group of units, $U(k\ltimes V)$, of $k\ltimes V$, is $k^\times\times V$. It follows that $k\ltimes V$ is a local ring, with Jacobson radical $J(k\ltimes V)$ equal to $\left\{0\right\}\times V$. Thus, up to isomorphisms, the module

\[S=k\ltimes V/J(k\ltimes V)\] 
 
\noindent is the only simple $k\ltimes V$-module and a $k\ltimes V$-module $M$ is simple if and only if $dim_kM=1$. Moreover, since $J(k\ltimes V)$ is clearly semisimple, the following equality
\[J(k\ltimes V)=Soc(k\ltimes V)\]
holds. It follows that $k\ltimes V$ is semiartinian with Goldie dimension $dim_kV$. It also follows that if $E_0$ denotes the minimal injective cogenerator in $k\ltimes V$, $E(S)$, then 
\[E(M)=E_0^{(dim_kSoc(M))}=E_0^{(GdimM)}\]
for every module $M$ with finite Goldie dimension. The following will be the main result of this section.

\begin{thm}
Let $k$ be a field. Let $V$ be a finite dimensional $k$-vector space.
\begin{enumerate}
\item If $k$ has characteristic $\neq 2$, then there exist indecomposable $k\ltimes V$-modules with Goldie dimension $m$ for every $m$ such that $1\leq m\leq 3dim_kV-2$.
\item If $k$ has characteristic $2$, then there exist indecomposable $k\ltimes V$-modules with Goldie dimension $m$ for every $m$ such that $1\leq m\leq dim_kV$, $dim_kV+\left\lfloor dim_kV/2\right\rfloor<m\leq 2dim_kV-1$, or $2dim_kV+\left\lfloor dim_kV/2\right\rfloor-1< m\leq 3dim_kV-2$.
\end{enumerate}
\end{thm}

\noindent We now reduce the proof of (3.1) to the proof of a series of lemmas. Given an ideal $I$ and a module $M$ we will write, in the next lemma and in the rest of this section, $I\ast M$ for the set $\bigcup_{a\in M}Ia$.

\begin{lem}
Let $I$ be an ideal
\begin{enumerate}
\item Let $\left\{M_\alpha :\alpha\in A\right\}$ be a collection of modules. If $M_\alpha\in\mathfrak{D}(I)$ for all $\alpha\in A$, then $\bigoplus_{\alpha\in A}M_\alpha\in\mathfrak{D}(I)$.
\item Let $M$ be a module. Let $N$ be a submodule of $M$. Suppose \[(I\ast M)\cap N=\left\{0\right\}\] If $M\in\mathfrak{D}(I)$, then $M/N\in\mathfrak{D}(I)$.
\end{enumerate}
\end{lem}

\begin{proof}
Let $\left\{M_\alpha:\alpha\in A\right\}$ be a collection of modules such that $M_\alpha\in\mathfrak{D}(I)$ for every $\alpha\in A$. Since $I^2M_\alpha=\left\{0\right\}$ for every $\alpha\in A$, we have that $I^2\bigoplus_{\alpha\in A}M_\alpha=\left\{0\right\}$. Now, let $a\in ann^*_I\bigoplus_{\alpha\in A}M_\alpha$. Let $r\in I\setminus\left\{0\right\}$ be such that $ra=0$. For each $\alpha\in A$, let $\pi_\alpha$ be the cannonical projection of $\bigoplus_{\alpha\in A}M_\alpha$ onto $M_\alpha$. Then $r\pi_\alpha(a)=0$ for all $\alpha\in A$, that is $\pi_\alpha(a)\in ann_I^*M_\alpha$ for all $\alpha\in A$. It follows that $\pi_\alpha(a)\in IM_\alpha$ for all $\alpha\in A$, that is $a\in I\bigoplus_{\alpha\in A}M_\alpha$. This concludes the proof of \textit{1}.
Now, let $M\in\mathfrak{D}(I)$. Let $N$ be a submodule of $M$ such that $(I\ast M)\cap N=\left\{0\right\}$. Since $I^2M=\left\{0\right\}$ it follows that $I^2M+N/N=N$. Let $a+N\in ann^*_IM/N$. Let $r\in I\setminus\left\{0\right\}$ be such that $ra\in N$. It follows, since $ra\in (I\ast M)\cap N$, that $ra=0$, that is, $a\in ann^*_IM$. It follows that $a\in IM$, that is $a+N\in IM+N/N$. This concludes the proof of \textit{2}. 
\end{proof}

\noindent Thus, given an ideal $I$, the domain of decomposition, $\mathfrak{D}(I)$, of $I$, is closed under direct sums, under quotients satisfying condition \textit{2} of (3.2), and by (2.4), under $I$-pure submodules. It follows that if $R$ is in the domain of decomposition of $I$, then every module that can be expressed as a quotient satisfying condition \textit{2} of (3.2) of some $I$-pure submodule of a free module, is also in the domain of decomposition of $I$. In particular, in this case, $I$ is a decomposition ideal of $P$ for every projective module $P$. The following lemma establishes, in the case where the base ring $R$ is local, a criterion under which $J(R)$ is a decomposition ideal of $R$.

\begin{lem}
Suppose $R$ is local. Then $J(R)\in\mathfrak{D}^{-1}(R)$ if and only if 

\[J(R)^2=\left\{0\right\}\]
\end{lem}

\begin{proof}
It is clear that if $J(R)\in\mathfrak{D}^{-1}(R)$, then $J(R)^2=\left\{0\right\}$. Now, if $a\in R\setminus J(R)$, then $a\in U(R)$. From this it easily follows that $ann^*_{J(R)}R\subseteq J(R)$. Thus, if we suppose that $J(R)^2=\left\{0\right\}$, then clearly $J(R)\in \mathfrak{D}^{-1}(R)$. This concludes the proof.
\end{proof}

\noindent Given a $k$-vector space $V$, it is immediate that

\[J(k\ltimes V)^2=Soc(k\ltimes V)^2=\left\{0\right\}\]

\noindent That is, by (3.3), $J(k\ltimes V)=Soc(k\ltimes V)$ is a decomposition ideal of $k\ltimes V$. It follows that $Soc(k\ltimes V)$ is a decomposition ideal of $L$ for every free $k\ltimes V$-module $L$. In particular, in the case in which $k=\mathbb{F}_2$ and $V=\mathbb{F}_2^2$, $Soc(\mathbb{F}_2\ltimes \mathbb{F}_2^2)$ is a decomposition ideal of $\mathbb{F}_2\ltimes\mathbb{F}_2^2$ and of $(\mathbb{F}_2\ltimes\mathbb{F}_2^2)^2$. Now, it is easily seen that while

\[\Gamma_{Soc(\mathbb{F}_2\ltimes\mathbb{F}_2^2)}(\mathbb{F}_2\ltimes\mathbb{F}_2^2)\]

\noindent has only one vertex, the graph

\[\Gamma_{Soc(\mathbb{F}_2\ltimes\mathbb{F}_2^2)}((\mathbb{F}_2\ltimes\mathbb{F}_2^2)^2)\]

\noindent has three connected components. Thus, while $cdim \mathbb{F}_2\ltimes\mathbb{F}_2^2=1$, we have that $cdim (\mathbb{F}_2\ltimes\mathbb{F}_2^2)^2=3$. We conclude that the function $cdim$ that associates to each module $M$ such that $\mathfrak{D}^{-1}(M)\neq\emptyset$, its combinatorial dimension $cdimM$, need not, in general, be additive. 

\

\noindent \textit{\textbf{Proof of 3.1}} Throughout the proof we will write $n$ instead of $dim_kV$ and we will identify each finite dimensional $k$-vector space $W$ with $k^{(dim_kW)}$. Let $m\leq n$. Let $W$ be a subspace of $V$ such that $dim_kW=n-m$. From the the correspondence theorem it follows that $k\ltimes V/\left\{0\right\}\oplus W$ is a local module with Goldie dimension $m$. Thus there exist indecomposable $k\ltimes V$-modules with Goldie dimension $m$ for every $m\leq n$.

Now, suppose $W$ is a $k$-vector space such that $n\leq dim_kW$. Let $i$ be such that $1\leq i\leq n-1$. We will denote by $W(i)$ the subspace of $W\oplus V$ generated by all vectors of the form:

\[(0_{dim_kW-n+i},v,-v,0_i)\]

\noindent with $v\in k^{n-i}$, and where $0_m$ denotes the $0$ vector in $k^m$ for any $m$. 

\noindent For each $i$ such that $1\leq i\leq n-1$ we will denote by $M_i$ the module

\[(k\ltimes V)^2/Soc(k\ltimes V)(i)\]

\noindent It is easily seen that $dim_kSoc(M_i)=n+i$, that is, $GdimM_i=n+i$. Now, $Soc(k\ltimes V)\ast (k\ltimes V)^2$ is equal to the set of all vectors of the form $(v,0)$, $(0,v)$, or $(v,v)$, with $v\in V$, while $Soc(k\ltimes V)(i)$ is equal to the set of all vectors of the form $(0_i,v,-v,0_i)$, with $v\in k^{n-i}$. From this it follows that if $k$ has characteristic $\neq 2$, or if $k$ has characteristic $2$ and $i> \left\lfloor n/2\right\rfloor$, the following equality

\[(Soc(k\ltimes V)\ast (k\ltimes V)^2)\cap Soc(k\ltimes V)(i)=\left\{0\right\}\]

\noindent holds. From this, from (3.2) and (3.3) and from the identities

\[Soc(k\ltimes V)^2=\left\{0\right\} \ \mbox{and} \ Soc(k\ltimes V)=J(k\ltimes V)\]

\noindent it follows that, in the cases mentioned, $M_i\in\mathfrak{D}(Soc(k\ltimes V))$.

Now, let $\alpha_j^i\in E_0^{n+i}$, $j\in\left\{0,1\right\}$, be such that $ann_{k\ltimes V}\alpha_j^i=\left\{0\right\}$ and such that for every $v\in V$

\[v\alpha_j^i=\left\{\begin{array}{ll}
(v,0_i) & \mbox{if $j=0$} \\
        &                  \\
(0_i,v) & \mbox{if $j=1$}\end{array} \right.\]
 
\noindent Then $M_i$ and the submodule

\[(k\ltimes V)\alpha_0^i+(k\ltimes V)\alpha_1^i\]

\noindent of $E_0^{n+i}$ have identical presentations and are thus isomorphic. We identify $M_i$ with this module. Now, if we identify cyclic submodules of $M_i$ having the same socle, then the set of cyclic submodules of $M_i$ generated by elements of $M_i\setminus (k\ltimes V)M_i$ is precisely the set formed by 

\[Soc(k\ltimes V)\alpha_0^i,\ Soc(k\ltimes V)\alpha_1^i,\ \mbox{and}\ Soc(k\ltimes V)\alpha_0^i+\alpha_1^i \]

\noindent with socles $V\oplus\left\{0_i\right\}$, $\left\{0_i\right\}\oplus V$, and the set $\left\{(v,0_i)+(0_i,v_i):v\in V\right\}$ respectively. Now, the existence of non-zero elements in 

\[Soc(k\ltimes V)\alpha_0^i\cap Soc(k\ltimes V)\alpha_1^i, \ \mbox{in} \ Soc(k\ltimes V)\alpha_0^i\cap Soc(k\ltimes V)\alpha_0^i+\alpha_1^i\]

\noindent and in

\[Soc(k\ltimes V)\alpha_1^i\cap Soc(k\ltimes V)\alpha_0^i+\alpha_1^i\]

\noindent is equivalent to the existence of non-trivial solutions in $k$ to the systems of homogenious linear equations

\[\begin{array}{lcl}x_1&=&0\\ \vdots&\vdots &\vdots \\ x_i&=&0\\x_{i+1}-y_1&=&0\\ \vdots&\vdots&\vdots \\x_n-y_{n-i}&=&0\\ -y_{n-i+1}&=&0\\ \vdots&\vdots &\vdots \\-y_n&=&0\end{array}, \begin{array}{lcl}x_1-y_1&=&0\\ \vdots&\vdots &\vdots \\x_i-y_i&=&0\\ x_{i+1}-y_{i+1}-y_1&=&0\\ \vdots&\vdots &\vdots \\ x_n-y_n-y_{n-i}&=&0\\ -y_{n-i+1}&=&0\\ \vdots &\vdots &\vdots \\ -y_n&=&0\end{array} \mbox{and} \begin{array}{lcl} -y_1&=&0\\ \vdots &\vdots &\vdots \\ -y_i&=&0\\x_1-y_1-y_{i+1}&=&0\\ \vdots&\vdots &\vdots \\ x_{n-i}-y_{n-i}-y_n&=&0\\x_{n_i+1}-y_{n-i+1}&=&0\\ \vdots&\vdots &\vdots \\ x_n-y_n&=&0\end{array}\]

\noindent respectively. All three linear systems are easily seen to have non-trivial solutions over $k$. It follows that $\Gamma_{Soc(k\ltimes V)}(M_i)$ is a complete graph for all $i$ considered. We conclude that there exist indecomposable $k\ltimes V$-modules of Goldie dimension $m$ for all $m$ such that $n\leq m\leq 2n-1$ in the case in which $k$ has characteristic $\neq 2$, and for all $m$ such that $n+\left\lfloor n/2\right\rfloor<m\leq 2n-1$ in the case in which $k$ has characteristic equal to $2$.  
 
Now, for each $i$ such that $1\leq i\leq n-1$, let $M_{n-1,i}$ denote the module

\[(M_{n-1}\oplus k\ltimes V)/Soc(k\ltimes V)(i)\]

\noindent Again, it is easily seen that $dim_kSoc(M_{n-1,i})=2n+i-1$, or equivalently, that $GdimM_{n-1,i}=2n+i-1$. Again, in this case, if $k$ has characteristic $\neq 2$, or $k$ has characteristic $2$ and $\left\lfloor n/2\right\rfloor<i$, we have

\[(Soc(k\ltimes V)\ast (M_{n-1}\oplus k\ltimes V))\cap Soc(M_{n-1})=\left\{0\right\}\]

\noindent from which it follows again that $M_{n-1,i}\in\mathfrak{D}(Soc(k\ltimes V))$ in the cases considered. It follows also that the cannonical projection of $M_{n-1}\oplus k\ltimes V$ onto $M_{n-1,i}$, restricted to $M_{n-1}$ is an isomorphism. We identify $M_{n-1}$ with the image of $M_{n-1}$ under this projection. Thus $M_{n-1}$ can be considered as a $k\ltimes V$-pure submodule of $M_{n-1,i}$, and by (2.4), under this identification, $\Gamma_{Soc(k\ltimes V)}(M_{n-1})$ is a complete subgraph of $\Gamma_{Soc(k\ltimes V)}(M_{n-1,i})$. 

Now, let $\alpha_{n-1,i}\in E_0^{2n+i-1}$ be such that $ann_{k\ltimes V}\alpha_{n-1,i}=\left\{0\right\}$ and such that 

\[v\alpha_{n-1,i}=(0_{n+i-i},v)\]

\noindent for all $v\in V$. If we identify $\alpha_j^{n-1}$ defined above, with $\alpha_j^{n-1}\times\left\{0_i\right\}$ in $E_0^{2n+i-1}$ for all $j=0,1$, then $M_{n-1,i}$ and the submodule

\[(k\ltimes V)\alpha_0^{n-1}+(k\ltimes V)\alpha_1^{n-1}+(k\ltimes V)\alpha_{n-1,i}\]

\noindent of $E_0^{2n+i-1}$ have identical presentations, and thus are isomorphic. We identify $M_{n-1,i}$ with this module. Now, again, identifying cyclic submodules of $M_{n-1,i}$ with the same socle, cyclic modules generated by elements of $M_{n-1,i}\setminus Soc(k\ltimes V)M_{n-1,i}$ are precisely cyclic modules generated by sums of $\alpha_0^{n-1}$, $\alpha_1^{n-1}$, and $\alpha_0^{n-1}+\alpha_1^{n-1}$ with $\alpha_{n-1,i}$ taken two at the time, that is, after identifying vertices with the same socle, the vertex set of $\Gamma_{Soc(k\ltimes V)}(M_{n-1,i})$ is the vertex set of $\Gamma_{Soc(k\ltimes V)}(M_{n-1})$, together with $(k\ltimes V)\alpha_{n-1,i}$ and all cyclic modules generated by sums of $\alpha_{n-1,i}$ and generators of vertices of $\Gamma_{Soc(k\ltimes V)}(M_{n-1})$. 

Now, the existence of non-zero elements in 

\[Soc(k\ltimes V)\alpha_1^{n-1}\cap Soc(k\ltimes V)\alpha_{n-1,i}\]

\noindent and in 

\[Soc(k\ltimes V)\alpha_1^{n-1}\cap Soc(k\ltimes V)\alpha_1^{n-1}+\alpha_{n-1,i}\]

\noindent is equivalent to the existence, in $k$, of non-trivial solutions to the systems of homogenious linear equations

\[\begin{array}{lcl}x_1&=&0\\ \vdots&\vdots&\vdots\\ x_i&=&0\\x_{i+1}-y_1&=&0\\ \vdots&\vdots&\vdots\\ x_n-y_{n-i}&=&0\\ -y_{n-i+1}&=&0\\ \vdots&\vdots&\vdots\\-y_n&=&0\end{array} \mbox{and} \begin{array}{lcl}x_1-y_1&=&0\\ \vdots&\vdots&\vdots\\ x_i-y_i&=&0\\x_{i+1}-y_{i+1}-y_1&=&0\\ \vdots&\vdots&\vdots\\ x_n-y_n-y_{n-i}&=&0\\ -y_{n-i+1}&=&0\\ \vdots&\vdots&\vdots\\-y_n&=&0\end{array}\]

\noindent respectively. It is easily seen that these two systems admit non-trivial solutions. It follows that $(k\ltimes V)\alpha_1^{n-1}$ is adjacent in $\Gamma_{Soc(k\ltimes V)}(M_{n-1,i})$ to vertices $(k\ltimes V)\alpha_{n-1,i}$ and $(k\ltimes V)\alpha_1^{n-1}+\alpha_{n-1,i}$. From this and from the fact that the graph $\Gamma_{Soc(k\ltimes V)}(M_{n-1})$ is connected it follows that $\Gamma_{Soc(k\ltimes V)}(M_{n-1,i})$ has at most three connected components, namely, the connected component generated by $(k\ltimes V)\alpha_0^{n-1}$, and possibly the connected component generated by $(k\ltimes V)\alpha_0^{n-1}+\alpha_{n-1,i}$ and the connected component generated by $(k\ltimes V)\alpha_0^{n-1}+\alpha_1^{n-1}+\alpha_{n-1,i}$. 

\noindent Thus, since all vertices in the connected component generated by $(k\ltimes V)\alpha_0^{n-1}+\alpha_{n-1,i}$ have the same socle, and all vertices in the connected component generated by $(k\ltimes V)\alpha_0^{n-1}+\alpha_1^{n-1}+\alpha_{n-1,i}$ have the same socle, the submodule of $M_{n-1,i}$ generated by these two components has Goldie dimension at most $2n$. It follows that if $i>1$, that is, if $GdimM_{n-1,i}>2n$, then the set $\left\{\alpha_0^{n-1},\alpha_{n-1,i}\right\}$ is fundamental in $M_{n-1,i}$ with respect to $Soc(k\ltimes V)$. It follows that in this case $fcdimM_{n-1,i}=1$. From this and from (2.13) we conclude that there exist indecomposable $k\ltimes V$-modules of Goldie dimension $m$ for all $m$ such that $2n<m\leq 3n-2$ when $k$ has characteristic $\neq 2$, and for all $m$ such that $2n+\left\lfloor n/2\right\rfloor -1< m\leq 3n-2$ when $k$ has characteristic $2$ and $n>2$. Now, the existence of non-zero elements in  $(k\ltimes V)\alpha_0^{n-1}+\alpha_{n-1,i}\cap(k\ltimes V)\alpha_1^{n-1}+\alpha_{n-1,i}$ is equivalent to the existence of non-trivial solutions, in $k$, to the system of homogenious linear equations

\[\begin{array}{lcl}x_1&=&0\\ \vdots&\vdots&\vdots \\x_{n-1}&=&0\\ x_n-y_1&=&0\\-y_2&=&0\\ \vdots&\vdots&\vdots \\-y_i&=&0\\ x_1-y_{i+1}-y_1&=&0\\ \vdots&\vdots&\vdots\\x_{n-i}-y_n-y_{n-i}&=&0\\x_{n-i+1}-y_{n-i+1}&=&0\\ \vdots&\vdots&\vdots\\x_n-y_n&=&0\end{array}\]

\noindent which in the case where $i=1$, has non-trivial solutions. It follows that, in this case, the graph $\Gamma_{Soc(k\ltimes V)}(M_{n-1,1})$ has at most two connected components, namely, the connected component generated by $(k\ltimes V)\alpha_0^{n-1}$, which in this case contains the vertex $(k\ltimes V)\alpha_0^{n-1}+\alpha_{n-1,1}$ and possibly the component generated by $(k\ltimes V)\alpha_0^{n-1}+\alpha_1^{n-1}+\alpha_{n-1,1}$. Again, since all vertices in the latter component have the same socle, the submodule of $M_{n-1,1}$ generated by this component has Goldie dimension $n$. It follows again that in this case the set $\left\{\alpha_0^{n-1},\alpha_{n-1,1}\right\}$ is fundamental in $M_{n-1,1}$, with respect to $Soc(k\ltimes V)$, from which we conclude that again $fcdimM_{n-1,1}=1$. From this and from (2.13) we conclude that in the case where $k$ has characteristic $\neq 2$ there exist indecomposable $k\ltimes V$-modules of Goldie dimension $2n$. This conlcudes the proof $\blacksquare$

\

\noindent The author does not know if there exist indecomposable $k\ltimes V$-modules with Goldie dimension in the intervals $dim_kV\leq m\leq dim_kV+\left\lfloor dim_kV/2\right\rfloor$ and $2dim_kV-1\leq m\leq 2dim_kV\left\lfloor dim_kV 2\right\rfloor-1$ when the field $k$ has characteristic $2$. There are cases in which there exist indecomposable $k\ltimes V$-modules with Goldie dimesion $\geq 3dim_kV-2$, nevertheless, the author does not know exactly when this happens. In the next section we formulate a few conjectures in this direction.

\section{Indecomposability in larger dimensions}

\noindent In this section we generalize the constructions performed in the course of the proof of (3.1). We present two related conjectures and we prove that these conjectures imply a generalized version of (3.1). We generalize the construction performed in the course of the proof of (3.1) as follows:

Let $k$ be a field and $V$ be a finite dimensional $k$-vector space. We will say that a, possibly empty, finite sequence of non-negative integers, $s=\left\{i_1,...,i_m\right\}$, is admissible, if $1\leq i_j\leq dim_kV-1$ and if $i_j\leq i_{j-1}$ for all $1\leq j\leq m$ in the case where the field $k$ has characteristic $\neq 2$. If $k$ has characteristic $2$, we further assume that all terms of an admissible sequence $s=\left\{i_1,...,i_m\right\}$ satisfy the inequality $\left\lfloor dim_kV/2\right\rfloor<i_j$.

\noindent For each admissible sequence $s$, we define the $k\ltimes V$-module $M_s$, recursively, as follows: 

\begin{enumerate}
\item We set $M_\emptyset=k\ltimes V$.
\item Suppose the module $M_s$ has been defined. Let $i_{m+1}$ be such that the sequence $s\cup\left\{i_{m+1}\right\}$ is admissible. We set

\[M_{s\cup\left\{i_{m+1}\right\}}=(M_s\oplus k\ltimes V)/Soc(M_s)(i_{m+1})\]

\end{enumerate}

\noindent Observe that for every $1\leq i \leq dim_kV-1$, $M_{\left\{i\right\}}$ is equal to $M_i$ and $M_{\left\{dim_kV-1,i\right\}}$ is equal to $M_{n-1,i}$ as they were defined in the course of the proof of (3.1). Doing induction on the length $m$ of the admissible sequence $s=\left\{i_1,...,i_m\right\}$ it is easily seen that

\begin{enumerate}
\item $dim_kSoc(M_s)=dim_kV+\sum_{j=1}^mi_j$ 
\item $(Soc(k\ltimes V)\ast M_{s\cup\left\{i_{m+1}\right\}})\cap Soc(M_s)(i_{m+1})=\left\{0\right\}$, for all $ i_{m+1}\leq i_m$.
\end{enumerate}

\noindent Observe that from 1 the identity
\[GdimM_s=dim_kV+\sum_{j=1}^mi_j\]

\noindent follows. Observe also that from 2, together with (3.2) and (3.3) it follows that

\[M_s\in\mathfrak{D}(Soc(k\ltimes V))\]

\noindent for every admissible sequence $s$. 

\noindent Given a, not necessarely finite, sequence $s$, with at least $m$ terms, and $n\leq m$, we denote by $s_n$ the $n$-th truncation of $s$. It follows, again, from 2 above, that for every admissible sequence $s$ of length $m$, and for every $n\leq m$, the module $M_{s_n}$ can be identified, up to isomorphisms, with a $Soc(k\ltimes V)$-pure submodule of $M_s$. We assume always that this identification has been performed.

\noindent Now, for every sequence $s=\left\{i_1,...,i_m\right\}$, of length $m$, let $\alpha_s\in E_0^{dim_kV+\sum_{j=1}^mi_j}$ be such that 

\[v\alpha_s=(0_{\sum_{j=1}^mi_j},v)\]

\noindent for every $v\in V$. We now define, recursively, for each admissible sequence $s=\left\{i_1,...,i_m\right\}$, the subset $\Sigma_s$ of $E_0^{dim_kV+\sum_{j=1}^mi_j}$ as follows:

\begin{enumerate}
\item We make $\Sigma_\emptyset$ to be $\left\{\alpha_\emptyset\right\}$.
\item Suppose the set $\Sigma_s$ has been defined. Let $i_{m+1}$ be such that the sequence $s\cup\left\{i_{m+1}\right\}$ is admissible. We make $\Sigma_{s\cup\left\{i_{m+1}\right\}}$ to be the set 

\[\Sigma_s\times \left\{0_{i_{m+1}}\right\}\cup\left\{\alpha_{s\cup\left\{i_{m+1}\right\}}\right\}\]

\end{enumerate}

\noindent It is easily seen that both the module $M_s$ defined above and the submodule of $E_0^{dim_kV+\sum_{j=1}^mi_j}$, generated by $\Sigma_s$ have identical presentations, and thus are isomorphic. Observe that the set $\Sigma_s$ associated to the sequence $s=\left\{i\right\}$ is equal to the set of generators of $M_i$ defined in the course of the proof of (3.1), and that the set $\Sigma_s$ associated to the sequence $s=\left\{n-1,i\right\}$ is equal to the set of generators of $M_{n-1,i}$ defined in the course of the proof of (3.1).

\noindent The vertex set $\left|\Gamma_{Soc(k\ltimes V)}(M_s)\right|$ of the $Soc(k\ltimes V)$-graph of cyclic modules, $\Gamma_{Soc(k\ltimes V)}(M_s)$, of $M_s$, is equal to the set, $\tilde{\Sigma_s}$, of sums of different elements of $\Sigma_s$. This observation, together with the recursive construction of the sets $\Sigma_s$ yields, after the computation of all systems of linear equations of the form 

\[va=vb, \ v\in V\setminus\left\{0\right\}\]

\noindent with $a,b\in \tilde{\Sigma_s}$, a recursive description of the $Soc(k\ltimes V)$-graph of cyclic modules of $M_s$.

\noindent Given an infinite sequence $s$, we say that $s$ is admissible if all its truncations are admissible. We conjecture the following.

\begin{conj}
Let $k$ be a field. Let $V$ be a finite dimensional $k$-vector space. If $dim_kV$ is sufficiently large, then there exist, an integer $\kappa_V$, an infinite admissible sequence $s$, and an infinite subsequence $t$ of $s$ such that $cdim M_{t_m}\leq \kappa_V$ for every $m$. 
\end{conj}

\begin{conj}
Let $k$ be a field. Let $V$ be a finite dimensional $k$-vector space. If $dim_kV$ is sufficently large, then there exist, an integer $\phi_V$, an infinite admissible sequence $s$, and an infinite subsequence $t$ of $s$ such that $M_{t_m}\in\mathfrak{F}(Soc(k\ltimes V))$ and such that $fcdim M_{t_m}\leq \phi_V$ for every $m$.
\end{conj}

\noindent Finally, the following proposition says that conjectures 4.1 and 4.2, imply the existence of indecomposable $k\ltimes V$-modules of arbitrarily large finite Goldie dimension.

\begin{prop}
Let $k$ be a field. Let $V$ be a finite dimensional $k$-vector space. Both conjectures 4.1 and 4.2 above imply that if $dim_kV$ is sufficently large, then, for each $n\geq 1$ there exists an indecomposable $k\ltimes V$-module $M$ such that $GdimM\geq n$. 
\end{prop}

\begin{proof}
Suppose conjecture 4.1 is true for $V$ with constant $\kappa_V$, infinite admissible sequence $s$, and infinite subsequence $t=\left\{i_1,i_2,...\right\}$. Let $n\geq 1$. Let $m$ be such that

\[n\kappa_V\leq dim_kV+\sum_{j=1}^mi_j\]

\noindent and such that $cdimM_{t_m}\leq \kappa_V$. Then, by the remarks made above $n\kappa_V\leq GdimM_{t_m}$. It follows, by (2.2) that $KS\ell(M_{t_m})\leq \kappa_V$. Let $M_{t_m}=\bigoplus_{i=1}^rM_i$ be an indecomposable direct sum decomposition of $M_{t_m}$. Then $r\leq \kappa_V$ and since 

\[n\kappa_V\leq GdimM_{t_m}=\sum_{j=1}^rGdimM_i\]

\noindent there exists $1\leq j\leq r$ such that $Gdim M_j\geq n$. Thus $M_j$ is a $k\ltimes V$-indecomposable module of Goldie dimension $\geq n$. The proof that conjecture 4.2 implies the claim is analogous. This concludes the proof.
\end{proof}

\section{Categorification}

\noindent In this section we present a functorial reformulation of the concept of graph of cyclic modules. We do this by defining families of concrete functors between suitable concrete categories in such a way that these constructions generalize the constructions made in section 2. We begin by setting our source and target categories. 

\noindent We define the category $\mathcal{C}$ as follows:

\begin{enumerate}
\item The class of objects $Ob_{\mathcal{C}}$ of $\mathcal{C}$ will be the class of all triples $(M,\Sigma,\Sigma')$ where $M$ is a module, $\Sigma$ is a subset of $M\setminus\left\{0\right\}$ and $\Sigma'$ is a subset of $\Sigma$.
\item Given two objects $(M,\Sigma,\Sigma')$ and $(N,\Lambda,\Lambda')$ in $\mathcal{C}$, the set of $\mathcal{C}$-morphisms $\Hom _{\mathcal{C}}((M,\Sigma,\Sigma'),(N,\Lambda,\Lambda'))$ will be the set of all $f\in \Hom_R(M,N)$ such that $f(\Sigma)\subseteq \Lambda$, and $f(\Sigma')\subseteq \Lambda'$.
\end{enumerate}

\noindent We make $\mathcal{C}$ into a concrete category over $R$-Mod with the obvious choice of forgetful functor. The category $\mathcal{C}$ will be the source category of our construction.

\noindent Now, denote by $G$ the category whose objects are graphs, with exactly one loop on each vertex, and their morphisms. We define the category $\mathcal{G}$ as follows:

\begin{enumerate}
\item The class of objects $Ob_{\mathcal{G}}$ of $\mathcal{G}$ will be the class of all pairs $(\Gamma,\Delta)$, where $\Gamma$ is an object in $G$, and $\Delta$ is a subset of the vertex set, $\left|\Gamma\right|$ of $\Gamma$.
\item Given two objects $(\Gamma,\Delta)$ and $(\Gamma',\Delta')$ of $\mathcal{G}$, the set $\Hom_{\mathcal{G}}((\Gamma,\Delta),(\Gamma',\Delta'))$ will be the set of all $f\in \Hom_G(\Gamma,\Gamma')$ such that $f(\Delta)\subseteq \Delta'$.
\end{enumerate}

\noindent We make $\mathcal{G}$ into a concrete category over $G$ with the obvious choice of forgetful functor. The category $\mathcal{G}$ will be the target of our construction.

\noindent Now, given an ideal $I$, we define the functor $\Gamma_I:\mathcal{C}\rightarrow \mathcal{G}$ as follows:
\begin{enumerate}
\item Let $(M,\Sigma,\Sigma')$ be an object in $\mathcal{C}$. We make $\Gamma_I(M,\Sigma,\Sigma')$ to be equal to $(\Gamma_I(M,\Sigma),\bigcup_{a\in \Sigma'}C_a^I)$.
\item Given two objects $(M,\Sigma,\Sigma')$ and $(N,\Lambda,\Lambda')$ in $\mathcal{C}$ and a $\mathcal{C}$-morphism $f\in \Hom_{\mathcal{C}}((M,\Sigma,\Sigma'),(N,\Lambda,\Lambda'))$, we make

\[\Gamma_I(f):\left|\Gamma_I(M,\Sigma,\Sigma')\right|\rightarrow\left|\Gamma_I(N,\Lambda,\Lambda')\right|\]

to be such that $\Gamma_I(f)(Ra)=Rf(a)$ for all $a\in \Sigma$, and we extend $\Gamma_I(f)$ to a morphism in $\mathcal{G}$
\end{enumerate}

\noindent Observe that $\Gamma_I$ thus defined is a well defined concrete functor between concrete categories $\mathcal{C}$ and $\mathcal{G}$. Observe also that the image of a triple of the form $(M,\Sigma,\emptyset)$ under $\Gamma_I$ can be naturally identified with the $I$-graph of cyclic modules, $\Gamma_I(M,\Sigma)$, of $\Sigma$, in $M$, studied in sections 2 and 3. We thus regard the collection of all functors $\Gamma_I$ as a categorification of the concept of graph of cyclic modules. 

\noindent We now proceed to the study of categories $\mathcal{C}$ and $\mathcal{G}$. The following proposition says that $C$ is a concretely complete category admitting concrete coproducts and concrete direct limits.

\begin{lem}
$\mathcal{C}$ is a concretely complete category. Further, $\mathcal{C}$ admits concrete coproducts and concrete direct limits.
\end{lem}

\begin{proof}
We prove first that $\mathcal{C}$ is concretely complete. We prove that $\mathcal{C}$ admits concrete products over $R$-Mod. Let $\left\{(M_\alpha,\Sigma_\alpha,\Sigma'_\alpha):\alpha\in A\right\}$ be a collection of objects in $\mathcal{C}$. Let $\pi_\alpha$ denote the cannonical projection of $\prod_{\alpha\in A} M_\alpha$ onto $M_\alpha$ for each $\alpha\in A$. We prove that the triple

\[(\prod_{\alpha\in A}M_\alpha,\prod_{\alpha\in A}\Sigma_\alpha,\prod_{\alpha\in A}\Sigma'_\alpha)\]

\noindent together with projections $\pi_\alpha$, $\alpha\in A$, is a product, in $\mathcal{C}$, for the collection $\left\{(M_\alpha,\Sigma_\alpha,\Sigma'_\alpha):\alpha\in A\right\}$. Observe first that for each $\alpha\in A$, projection $\pi_\alpha$ is a morphism in $\mathcal{C}$. Now, let $(N,\Lambda,\Lambda')$ be an object in $\mathcal{C}$ and let, for each $\alpha\in A$, $\mu_\alpha:(N,\Lambda,\Lambda')\rightarrow (M_\alpha,\Sigma_\alpha,\Sigma'_\alpha)$, be a morphism in $\mathcal{C}$. There exists a unique morphism $\mu:N\rightarrow\prod_{\alpha\in A}M_\alpha$, in $R$-Mod, such that $\mu_\alpha\mu=\pi_\alpha$ for all $\alpha\in A$. From this equations and from the fact that $\mu_\alpha(\Lambda)\subseteq \Sigma_\alpha$ and $\mu_\alpha(\Lambda')\subseteq \Sigma'_\alpha$ for all $\alpha\in A$ it follows that $\mu(\Lambda)\subseteq \prod_{\alpha\in A}\Sigma_\alpha$, and that $\mu(\Lambda')\subseteq \prod_{\alpha\in A}\Sigma'_\alpha$, that is, $\mu$ is a morphism in $\mathcal{C}$. We conclude that $\mathcal{C}$ admits concrete products.

We now prove that $\mathcal{C}$ admits concrete equalizers. Let $(M,\Sigma,\Sigma')$ and $(N,\Lambda,\Lambda')$ be two objects in $\mathcal{C}$, and let $f,g:(M.\Sigma,\Sigma')\rightarrow (N,\Lambda,\Lambda')$ be morphisms in $\mathcal{C}$. The module $K=ker(f-g)$, together with inclusion $\iota$ of $K$ in $M$, is a coequalizer for the pair $f,g$ in $R$-Mod. We prove that the triple

\[(K,K\cap\Sigma,K\cap\Sigma')\]

\noindent together with $\iota$, is a coequalizer for the pair $f,g$ in $\mathcal{C}$. Observe first that $\iota$ is a morphism in $\mathcal{C}$ such that $f\iota=g\iota$. Now, let $(L,\Delta,\Delta')$ be an object in $\mathcal{C}$ and $\mu:(L,\Delta,\Delta')\rightarrow (M,\Sigma,\Sigma')$ be a morphism in $\mathcal{C}$ such that $f\mu=g\mu$. Since $K$, together with $\iota$, is an equalizer, in $R$-Mod, of the pair $f,g$, there exists a unique morphism $\nu:L\rightarrow K$ such that $\iota\nu=\mu$. From this equation, together with the fact that $\mu(\Delta)\subseteq \Sigma$ and $\mu(\Delta')\subseteq \Sigma'$, it follows that $\nu(\Delta)\subseteq K\cap\Sigma$ and $\nu(\Delta')\subseteq K\cap \Sigma'$, that is, $\nu$ is a morphism in $\mathcal{C}$. We conclude that $\nu$ admits concrete equalizers. From this and from [1; 12.3] we conclude that $\mathcal{C}$ is concretely complete over $R$-Mod. 

We now prove that the concrete category $\mathcal{C}$ admits concrete coproducts. Let $\left\{(M_\alpha,\Sigma_\alpha,\Sigma'_\alpha):\alpha\in A\right\}$ be a collection of objects in $\mathcal{C}$. Let $\iota_\alpha$ denote the cannonical inclusion of $M_\alpha$ into $\bigoplus_{\alpha\in A}M_\alpha$ for each $\alpha\in A$. We prove that the triple

\[(\bigoplus_{\alpha\in A}M_\alpha,\bigcup_{\alpha\in A}\iota_\alpha(\Sigma_\alpha),\bigcup_{\alpha\in A}\iota_\alpha(\Sigma'_\alpha))\]

\noindent together with inclusions $\iota_\alpha$ is a coproduct, in $\mathcal{C}$, for the family of objects $\left\{(M_\alpha,\Sigma_\alpha,\Sigma'_\alpha):\alpha\in A\right\}$. Observe that $\iota_\alpha$ is a morphism in $\mathcal{C}$ for all $\alpha\in A$. Now, let $(N,\Lambda,\Lambda')$ be an object in $\mathcal{C}$ and $\mu_\alpha:(M_\alpha,\Sigma_\alpha,\Sigma'_\alpha)\rightarrow(N,\Lambda,\Lambda')$, $\alpha\in A$, be morphisms in $\mathcal{C}$. There exists a unique morphism, in $R$-Mod, $\iota:\bigoplus_{\alpha\in A}M_\alpha\rightarrow N$ such that $\iota\iota_\alpha=\mu_\alpha$ for all $\alpha\in A$. It follows, form these equations, together with the fact that $\mu_\alpha(\Sigma_\alpha)\subseteq \Lambda$ and $\mu_\alpha(\Sigma'_\alpha)\subseteq \Lambda'$ for all $\alpha\in A$, that $\iota(\bigcup_{\alpha\in A}\iota_\alpha(\Sigma_\alpha))\subseteq \Lambda$ and that $\iota(\bigcup_{\alpha\in A}\iota_\alpha(\Sigma'_\alpha))\subseteq \Lambda'$, that is, $\iota$ is a morphism in $\mathcal{C}$. We conclude that $\mathcal{C}$ admits concrete coproducts.

Finally, we prove that every direct system in $\mathcal{C}$ admits concrete direct limits in $\mathcal{C}$. Let $\left\{(M_\alpha,\Sigma_\alpha,\Sigma'_\alpha),\varphi_{\alpha,\beta}:\alpha,\beta\in A\right\}$ be a direct system in $\mathcal{C}$. We prove that the triple 

\[(\varinjlim M_\alpha,\bigcup_{\alpha\in A}\varphi_{\alpha,\infty}(\Sigma_\alpha),\bigcup_{\alpha\in A}\varphi_{\alpha,\infty}(\Sigma'_\alpha))\]

\noindent together with morphisms $\varphi_{\alpha,\infty}$ is a direct limit, in $\mathcal{C}$, for the direct system $\left\{(M_\alpha,\Sigma_\alpha,\Sigma'_\alpha),\varphi_{\alpha,\beta}:\alpha,\beta\in A\right\}$. Observe first that morphism $\varphi_{\alpha,\infty}$ is a morphism in $\mathcal{C}$ for every $\alpha\in A$. Now, let $(N,\Lambda,\Lambda')$ be an object in $\mathcal{C}$ and $\mu_\alpha:(M_\alpha,\Sigma_\alpha,\Sigma'_\alpha)\rightarrow(N,\Lambda,\Lambda')$, $\alpha\in A$, be morphisms in $\mathcal{C}$ such that $\mu_\alpha=\mu_\beta\varphi_{\alpha,\beta}$ for all $\alpha,\beta\in A$ such that $\alpha\leq\beta$. There exists a unique morphism $\mu:\varinjlim M_\alpha\rightarrow N$ in $R$-Mod such that $\mu\varphi_{\alpha,\infty}=\mu_\alpha$ for all $\alpha\in A$. From these equations and from the fact that $\mu_\alpha(\Sigma_\alpha)\subseteq \Lambda_\alpha$ and $\mu_\alpha(\Sigma'_\alpha)\subseteq \Lambda'$ for all $\alpha\in A$ it follows that $\mu(\bigcup_{\alpha\in A}\varphi_{\alpha,\infty}(\Sigma_\alpha))\subseteq \Lambda$ and that $\mu(\bigcup_{\alpha\in A}\varphi_{\alpha,\infty}(\Sigma'_\alpha))\subseteq \Lambda'$, that is, $\mu$ is a morphism in $\mathcal{C}$. We conclude that $\mathcal{C}$ admits concrete direct limits. This concludes the proof.
\end{proof}

\noindent Given a nonempty collection of objects $\left\{(M_\alpha,\Sigma_\alpha,\Sigma'_\alpha):\alpha\in A\right\}$, in $\mathcal{C}$, we will denote by

\[\prod_{\alpha\in A}(M_\alpha,\Sigma_\alpha,\Sigma'_\alpha) \ \mbox{and} \ \coprod_{\alpha\in A}(M_\alpha,\Sigma_\alpha,\Sigma'_\alpha)\]

\noindent the product and coproduct respectively, in $\mathcal{C}$, of $\left\{(M_\alpha,\Sigma_\alpha,\Sigma'_\alpha):\alpha\in A\right\}$. Given a direct system $\left\{(M_\alpha,\Sigma_\alpha,\Sigma'_\alpha),\varphi_{\alpha,\beta}:\alpha,\beta\in A\right\}$ in $\mathcal{C}$, we will denote by

\[\varinjlim (M_\alpha,\Sigma_\alpha,\Sigma'_\alpha)\]

\noindent the direct limit, in $\mathcal{C}$, of $\left\{(M_\alpha,\Sigma_\alpha,\Sigma'_\alpha),\varphi_{\alpha,\beta}:\alpha,\beta\in A\right\}$.
Observe that in general, not every pair of morphisms in $\mathcal{C}$ admits concrete coequalizer in $\mathcal{C}$. To see this observe that if $k$ is a field of characteristic $\neq 2$, then $\left\{0\right\}$ is coequalizer, in $k$-Mod, of the pair $id_k,-id_k$. Thus, in this case, the pair $id_k,-id_k$ considered as endomorphisms of $(k,k\setminus\left\{0\right\},\emptyset)$, in $\mathcal{C}$, does not admit concrete coequalizer in $\mathcal{C}$.

\begin{lem}
$\mathcal{G}$ is concretely complete and concretely cocomplete.
\end{lem}

\begin{proof}
We prove first that $\mathcal{G}$ is concretely complete. We prove that $\mathcal{G}$ admits concrete products. Let $\left\{(\Gamma_\alpha,\Delta_\alpha):\alpha\in A\right\}$ be a collection of objects in $\mathcal{G}$. Denote by $\prod_{\alpha\in A}\Gamma_\alpha$, the graph defined as follows:

\begin{enumerate}
\item $\left|\prod_{\alpha\in A}\Gamma_\alpha\right|=\prod_{\alpha\in A}\left|\Gamma_\alpha\right|$.
\item Given $a,b\in\prod_{\alpha\in A}\left|\Gamma_\alpha\right|$, $a$ and $b$ are adjacent in $\prod_{\alpha\in A}\Gamma_\alpha$ if $\pi_\alpha(a)$ and $\pi_\alpha(b)$ are adjacent in $\Gamma_\alpha$ for all $\alpha\in A$, where $\pi_\alpha$ denotes the cannonical projection of $\prod_{\alpha\in A}\left|\Gamma_\alpha\right|$ onto $\left|\Gamma_\alpha\right|$ for all $\alpha\in A$
\end{enumerate}

\noindent It is easily seen that the graph $\prod_{\alpha\in A}\Gamma_\alpha$, together with projections $\pi_\alpha$, $\alpha\in A$, is a product, in $G$, for the family $\left\{\Gamma_\alpha:\alpha\in A\right\}$. We prove that the pair

\[(\prod_{\alpha\in A}\Gamma_\alpha,\prod_{\alpha\in A}\Delta_\alpha)\]

\noindent together with projections $\pi_\alpha$, $\alpha\in A$, is a product, in $\mathcal{G}$, for the family $\left\{(\Gamma_\alpha,\Delta_\alpha):\alpha\in A\right\}$. Observe first that $\pi_\alpha$ is a morphism in $\mathcal{G}$ for all $\alpha\in A$. Now, let $(\Gamma',\Delta')$ be an object in $\mathcal{G}$ and let $\mu_\alpha:(\Gamma',\Delta')\rightarrow (\Gamma_\alpha,\Delta_\alpha)$, $\alpha\in A$, be morphisms in $\mathcal{G}$. Since $\prod_{\alpha\in A}\Gamma_\alpha$ is a product, in $G$, for the family $\left\{\Gamma_\alpha:\alpha\in A\right\}$, there exists a unique morphism $\mu:\Gamma'\rightarrow \prod_{\alpha\in A}\Gamma_\alpha$ in $G$ such that $\mu_\alpha=\pi_\alpha\mu$ for all $\alpha\in A$. From these equations and from the fact that $\mu_\alpha(\Delta')\subseteq \Delta_\alpha$ it follows that $\mu(\Delta')\subseteq \prod_{\alpha\in A}\Delta_\alpha$, that is, $\mu$ is a morphism in $\mathcal{G}$. We conclude that $\mathcal{G}$ admits concrete products.

We prove now that $\mathcal{G}$ admits concrete equalizers. Let $(\Gamma,\Delta)$ and $(\Gamma',\Delta')$ be two objects in $\mathcal{G}$. Let $f,g:(\Gamma,\Delta)\rightarrow (\Gamma',\Delta')$ be two morphisms in $\mathcal{G}$. Denote by $\Gamma_{f,g}$ the graph defined as follows:

\begin{enumerate}
\item $\left|\Gamma_{f,g}\right|=\left\{a\in \left|\Gamma\right|:f(a)=g(a)\right\}$
\item Let $a,b\in \left|\Gamma_{f,g}\right|$. Then $a$ and $b$ will be adjacent in $\Gamma_{f,g}$ if $a$ and $b$ are adjacent in $\Gamma$
\end{enumerate}

\noindent That is, $\Gamma_{f,g}$ is the subgraph of $\Gamma$ generated by the equalizer, in the category of sets and functions, of the pair $f,g$. It is easily seen that $\Gamma_{f,g}$, together with the cannonical inclusion, $\iota$, of $\Gamma_{f,g}$ into $\Gamma$, is an equalizer, in $G$, of the pair $f,g$. We prove that the pair

\[(\Gamma_{f,g},\Delta\cap \left|\Gamma_{f,g}\right|)\]
 
\noindent together with the inclusion $\iota$ is an equalizer, in $\mathcal{G}$, for the pair $f,g$. Observe first that $\iota$ is a morphism in $\mathcal{G}$. Let $(\Gamma'',\Delta'')$ be an object in $\mathcal{G}$. Let $\mu:(\Gamma'',\Delta'')\rightarrow (\Gamma,\Delta)$ be a morphism in $\mathcal{G}$ such that $f\mu=g\mu$. Since $\Gamma_{f,g}$, together with $\iota$, is a coequalizer, in $G$, for the pair $f,g$, then there exists a unique morphism, $\nu:\Gamma''\rightarrow \Gamma_{f,g}$ such that $\iota\nu=\mu$. Clearly $\nu(\Delta'')\subseteq \Delta\cap \left|\Gamma_{f,g}\right|$, that is, $\nu$ is a morphism in $\mathcal{G}$. We conclude that $\mathcal{G}$ admits concrete equalizers. By this and by [] we conclude that the category $\mathcal{G}$ is concretely complete.

We prove now that $\mathcal{G}$ is concretely cocomplete. We first prove that $\mathcal{G}$ admits concrete coproducts. Let $\left\{(\Gamma_\alpha,\Delta_\alpha):\alpha\in A\right\}$ be a collection of objects in $\mathcal{G}$. Denote by $\coprod_{\alpha\in A}\Gamma_\alpha$ the graph defined as follows:

\begin{enumerate}
\item $\left|\coprod_{\alpha\in A}\Gamma_\alpha\right|=\coprod_{\alpha\in A}\left|\Gamma_\alpha\right|$.
\item Let $a,b\in \coprod_{\alpha\in A}\left|\Gamma_\alpha\right|$. Then $a$ and $b$ will be adjacent in $\coprod_{\alpha\in A}\Gamma_\alpha$ if there exists $\alpha\in A$ such that $a$ and $b$ are adjacent in $\Gamma_\alpha$.
\end{enumerate}

\noindent where the coproduct symbol on the right hand side of 1 is taken in the category of sets and functions. Thus, $\coprod_{\alpha\in A}\Gamma_\alpha$ is defined by the disjoint union of vertex sets and adjacency relations of graphs in the collection $\left\{\Gamma_\alpha:\alpha\in A\right\}$. Denote, for each $\alpha\in A$, by $\iota_\alpha$, the cannonical inclusion of $\Gamma_\alpha$ into $\coprod_{\alpha\in A}\Gamma_\alpha$. It is easily seen that $\coprod_{\alpha\in A}\Gamma_\alpha$, togehter with cannonical inclusions $\iota_\alpha$, $\alpha\in A$, is a coproduct, in $\mathcal{G}$, of the collection $\left\{\Gamma_\alpha:\alpha\in A\right\}$. We prove that 

\[(\coprod_{\alpha\in A}\Gamma_\alpha,\bigcup_{\alpha\in A}\iota_\alpha(\Delta_\alpha))\]

\noindent together with morphisms $\iota_\alpha$, $\alpha\in A$, is a coproduct, in $\mathcal{G}$, of the collection $\left\{(\Gamma_\alpha,\Delta_\alpha):\alpha\in A\right\}$. Observe first that $\iota_\alpha$ is a morphism in $\mathcal{G}$ for all $\alpha\in A$. Now, let $(\Gamma',\Delta')$ be an object in $\mathcal{G}$. Let $\mu_\alpha:(\Gamma_\alpha,\Delta_\alpha)\rightarrow (\Gamma',\Delta')$, $\alpha\in A$, be morphisms in $\mathcal{G}$. Since $\coprod_{\alpha\in A}\Gamma_\alpha$, together with inclusions $\iota_\alpha$, $\alpha\in A$, is a coproduct, in $G$, for the family $\left\{(\Gamma_\alpha,\Delta_\alpha):\alpha\in A\right\}$, then there exists a unique morphism $\iota:\coprod_{\alpha\in A}\Gamma_\alpha\rightarrow \Gamma'$ in $G$ such that $\iota\iota_\alpha=\mu_\alpha$ for all $\alpha\in A$. It is easily seen, from these equations, that $\iota(\bigcup_{\alpha\in A}\iota_\alpha(\Delta_\alpha))\subseteq \Delta'$, that is, $\iota$ is a morphism in $\mathcal{G}$. We conclude that $\mathcal{G}$ admits concrete coproducts.

Finally, we prove that $\mathcal{G}$ admits concrete coequalizers. Again, let $(\Gamma,\Delta)$ and $(\Gamma',\Delta')$ be two objects in $\mathcal{G}$ and $f,g:(\Gamma,\Delta)\rightarrow(\Gamma',\Delta')$ be two morphisms in $\mathcal{G}$. Denote by $\Gamma^{f,g}$ the graph defined as follows:

\begin{enumerate}
\item $\left|\Gamma^{f,g}\right|$ will be the set of equivalence classes of $\left|\Gamma'\right|$ under the minimal equivalence relation identifying $f(a)$ with $g(a)$ for every $a\in \left|\Gamma\right|$.
\item Let $[a],[b]\in \left|\Gamma^{f,g}\right|$. Then $[a]$ and $[b]$ will be adjacent in $\Gamma^{f,g}$ is the exist $a'\in [a]$ and $b'\in [b]$ such that $a'$ and $b'$ are adjacent in $\Gamma'$.
\end{enumerate}

\noindent That is, $\Gamma^{f,g}$ is defined as the graph cogenerated, in $\Gamma'$, by the equivalence relation defining the set $\left|\Gamma^{f,g}\right|$. We will denote by $\pi$ the cannonical projection of $\Gamma'$ onto $\Gamma^{f,g}$. $\pi$ thus defined extends to a morphism in $G$ from $\Gamma'$ to $\Gamma^{f,g}$ such that $\pi f=\pi g$. It is easily seen the $\Gamma^{f,g}$, together with projection $\pi$, is a coequalizer for the pair $f,g$ in $G$. We will prove that the pair

\[(\Gamma^{f,g},\pi(\Delta'))\]

\noindent together with projection $\pi$, is a coequalizer, in $\mathcal{G}$, for the pair $f,g$. Observe first that clearly $\pi$ is a morphism in $\mathcal{G}$. Now, let $(\Gamma'',\Delta'')$ be an object in $\mathcal{G}$ and $\mu:(\Gamma',\Delta')\rightarrow (\Gamma'',\Delta'')$ be morphism in $\mathcal{G}$ such that $\mu f=\mu g$. Since $\Gamma^{f,g}$, together with projection $\pi$, is a coequalizer, in $G$, for the pair $f,g$, there exists a unique morphism $\nu:\Gamma^{f,g}\rightarrow \Gamma''$ such that $\nu\pi=\mu$. It follows, from this equation, that $\nu(\pi(\Delta'))\subseteq \mu(\Delta'')$, that is, $\nu$ is a morphism in $\mathcal{G}$. We conclude that $\mathcal{G}$ admits concrete coequalizers. Again by [1; 12.3] the category $\mathcal{G}$ is concretely cocomplete. This concludes the proof.

\end{proof}

\noindent Given a nonempty collection of objects $\left\{(\Gamma_\alpha,\Delta_\alpha):\alpha\in A\right\}$, in $\mathcal{G}$, we will denote by

\[\prod_{\alpha\in A}(\Gamma_\alpha,\Delta_\alpha) \ \mbox{and} \ \coprod_{\alpha\in A}(\Gamma_\alpha,\Delta_\alpha)\]

\noindent the product and the coproduct respectively, in $\mathcal{G}$, of $\left\{(\Gamma_\alpha,\Delta_\alpha):\alpha\in A\right\}$. Observe that it follows from (5.2) that in category $\mathcal{G}$ every direct system admits a direct limit. Given a direct system $\left\{(\Gamma_\alpha,\Delta_\alpha),\varphi_{\alpha,\beta}:\alpha,\beta\in A\right\}$ in $\mathcal{G}$, we will denote by

\[\varinjlim (\Gamma_\alpha,\Delta_\alpha)\]

\noindent the direct limit, in $\mathcal{G}$, of $\left\{(\Gamma_\alpha,\Delta_\alpha),\varphi_{\alpha,\beta}:\alpha,\beta\in A\right\}$. 

We now proceed to the study of functors of the form $\Gamma_I$. The next proposition says that functors of this type preserve direct limits of direct systems in $\mathcal{C}$. We first prove the following lemma. Recall first that the union, $\bigcup_{\alpha\in A}\Gamma_\alpha$, of a collection of graphs $\left\{\Gamma_\alpha:\alpha\in A\right\}$ is defined as the followin graph:
\begin{enumerate}
\item $\left|\bigcup_{\alpha\in A}\Gamma_\alpha\right|=\bigcup_{\alpha\in A}\left|\Gamma_\alpha\right|$.
\item Given $a,b\in \bigcup_{\alpha\in A}\left|\Gamma_\alpha\right|$, $a,b$ will be adjacent in $\bigcup_{\alpha\in A}\Gamma_\alpha$ if there exists $\alpha\in A$ such that $a,b\in \left|\Gamma_\alpha\right|$ and $a,b$ are adjacent in $\Gamma_\alpha$.
\end{enumerate}

\noindent Observe that for each $\alpha\in A$, the graph $\Gamma_\alpha$ is contained, as a subgraph, in $\bigcup_{\alpha\in A}\Gamma_\alpha$, and that $\bigcup_{\alpha\in A}\Gamma_\alpha$ is minimal with respect to this property.  

\begin{lem}
Let $I$ be an ideal. Let $M$ be a module. Let $\left\{M_\alpha:\alpha\in A\right\}$ be a collection of submodules of $M$. Let $\left\{\Sigma_\alpha:\alpha\in A\right\}$ be a collection of sets such that $\Sigma_\alpha\subseteq M_\alpha\setminus\left\{0\right\}$ for all $\alpha\in A$. If both sets $\left\{M_\alpha:\alpha\in A\right\}$ and $\left\{\Sigma_\alpha:\alpha\in A\right\}$ are directed with respect to the order given by inclusion in $M$, then 

\[\Gamma_I(\bigcup_{\alpha\in A}M_\alpha,\bigcup_{\alpha\in A}\Sigma_\alpha)=\bigcup_{\alpha\in A}\Gamma_I(M_\alpha,\Sigma_\alpha)\]
\end{lem}

\begin{proof}
Let $\left\{M_\alpha:\alpha\in A\right\}$ be a collection of submodules of $M$ and let $\left\{\Sigma_\alpha:\alpha\in A\right\}$ be a collection of sets such that $\Sigma_\alpha\subseteq M_\alpha\setminus\left\{0\right\}$ for all $\alpha\in A$. Suppose both sets $\left\{M_\alpha:\alpha\in A\right\}$ and $\left\{\Sigma_\alpha:\alpha\in A\right\}$ are directed. Clearly 

\[\left|\Gamma_I(\bigcup_{\alpha\in A}M_\alpha,\bigcup_{\alpha\in A}\Sigma_\alpha)\right|=\left|\bigcup_{\alpha\in A}\Gamma_I(M_\alpha,\Sigma_\alpha)\right|\]

\noindent Thus we only need to prove that the corresponding adjacency relations are equal. Now, since $\Gamma_I(M_\alpha,\Sigma_\alpha)$ is a subgraph of $\Gamma_I(\bigcup_{\alpha\in A}M_\alpha,\bigcup_{\alpha\in A}\Sigma_\alpha)$ for every $\alpha\in A$, the adjacency relation in $\bigcup_{\alpha\in A}\Gamma_I(M_\alpha,\Sigma_\alpha)$ is contained in the adjacency relation of $\Gamma_I(\bigcup_{\alpha\in A}M_\alpha,\bigcup_{\alpha\in A}\Sigma_\alpha)$. Now, let $a,b\in\bigcup_{\alpha\in A}\Sigma_\alpha$. Suppose $Ra$ and $Rb$ are adjacent in $\Gamma_I(\bigcup_{\alpha\in A}M_\alpha,\bigcup_{\alpha\in A}\Sigma_\alpha)$. Let $\gamma\in A$ be such that $a,b\in\Sigma_\gamma$. Then $Ra$ and $Rb$ are adjacent vertices in $\Gamma_I(M_\gamma,\Sigma_\gamma)$. It follows that $Ra$ and $Rb$ are adjacent vertices of $\bigcup_{\alpha\in A}\Gamma_I(M_\alpha,\Sigma_\alpha)$. We conclude that both graphs under consideration are equal. This concludes the proof.
\end{proof}

\begin{prop}
Let $I$ be an ideal. Let $\left\{(M_\alpha,\Sigma_\alpha,\Sigma'_\alpha),\varphi_{\alpha,\beta}:\alpha,\beta\in A\right\}$ be a direct system in $\mathcal{C}$. Then 

\[\varinjlim \Gamma_I(M_\alpha,\Sigma_\alpha,\Sigma'_\alpha)=\Gamma_I(\varinjlim(M_\alpha,\Sigma_\alpha,\Sigma'_\alpha))\]
\end{prop}

\begin{proof}
Let $\left\{(M_\alpha,\Sigma_\alpha,\Sigma'_\alpha),\varphi_{\alpha,\beta}:\alpha,\beta\in A\right\}$ be a direct system in $\mathcal{C}$. We prove that 

\[\Gamma_I(\varinjlim (M_\alpha,\Sigma_\alpha,\Sigma'_\alpha))\]

\noindent together with morphisms $\Gamma_I(\varphi_{\alpha,\infty})$, $\alpha\in A$, is a direct limit, in $\mathcal{C}$, for the system $\left\{\Gamma_I(M_\alpha,\Sigma_\alpha,\Sigma'_\alpha),\Gamma_I(\varphi_{\alpha,\beta}):\alpha,\beta\in A\right\}$. 

From the fact that 

\[\varphi_{\beta,\infty}\varphi_{\alpha,\beta}=\varphi_{\alpha,\infty}\]

\noindent for every $\alpha,\beta\in A$ such that $\alpha\leq\beta$ it follows that 

\[\Gamma_I(\varphi_{\beta,\infty})\Gamma_I(\varphi_{\alpha,\beta})=\Gamma_I(\varphi_{\alpha,\infty})\]

\noindent for all $\alpha,\beta\in A$ such that $\alpha\leq\beta$. Now, let $(\Gamma,\Delta)$ be an object of $\mathcal{G}$ and $\mu_\alpha:\Gamma_I(M_\alpha,\Sigma_\alpha,\Sigma'_\alpha)\rightarrow (\Gamma,\Delta)$ be morphisms in $\mathcal{G}$. From the fact that the set of connected components of $\Gamma_I(\bigcup_{\alpha\in A}M_\alpha,\bigcup_{\alpha\in A}\Sigma_\alpha)$, generated by $\bigcup_{\alpha\in A}\varphi_{\alpha,\infty}(\Sigma'_\alpha)$, is equal to the  set of connected components of $\bigcup_{\alpha\in A}\Gamma_I(M_\alpha,\Sigma_\alpha)$, generated by $\bigcup_{\alpha\in A}\Gamma_I(\varphi_{\alpha,\infty})(\Sigma'_\alpha)$, from (5.3), and from the equation

\[\varinjlim (M_\alpha,\Sigma_\alpha,\Sigma'_\alpha)=(\bigcup_{\alpha\in A}\varphi_{\alpha,\infty}(M_\alpha),\bigcup_{\alpha\in A}\varphi_{\alpha,\infty}(\Sigma_\alpha),\bigcup_{\alpha\in A}\varphi_{\alpha,\infty}(\Sigma'_\alpha))\]

\noindent It follows that

\[\Gamma_I(\varinjlim(M_\alpha,\Sigma_\alpha,\Sigma'_\alpha))=\bigcup_{\alpha\in A}\Gamma_I(\varphi_{\alpha,\infty})(\Gamma_I(M_\alpha,\Sigma_\alpha,\Sigma'_\alpha))\]

\noindent The right hand side of this equation is equal to 

\[\bigcup_{\alpha\in A}\Gamma_I(\varphi_{\alpha,\infty}(M_\alpha),\varphi_{\alpha,\infty}(\Sigma_\alpha),\varphi_{\alpha,\infty}(\Sigma'_\alpha))\]

\noindent Thus, we define $\nu:\varinjlim \Gamma_I(M_\alpha,\Sigma_\alpha,\Sigma'_\alpha)\rightarrow (\Gamma,\Delta)$ as follows: Let $R\varphi_{\alpha,\infty}(a)\in \left|\Gamma_I(M_\alpha,\Sigma_\alpha)\right|$ with $a\in\Sigma_\alpha$. We make $\nu(R\varphi_{\alpha,\infty}(a))$ to be equal to $\mu_\alpha(Ra)$ and we extend $\nu$ to a morphism in $G$. Clearly, $\nu$ thus defined, is a morphism in $\mathcal{G}$, it satisfies the equations $\nu\Gamma_I(\varphi_{\alpha,\infty})=\mu_\alpha$, and $\nu$ is unique with respect to this property. This concludes the proof. 

\end{proof}

\noindent The next proposition says that functors $\Gamma_I$ preserve coproducts in $\mathcal{C}$.

\begin{prop}
Let $I$ be an ideal. Let $\left\{(M_\alpha,\Sigma_\alpha,\Sigma'_\alpha):\alpha\in A\right\}$ be a non-empty collection of objects in $\mathcal{C}$. Then

\[\coprod_{\alpha\in A}\Gamma_I(M_\alpha,\Sigma_\alpha,\Sigma'_\alpha)=\Gamma_I(\coprod_{\alpha\in A}(M_\alpha,\Sigma_\alpha,\Sigma'_\alpha))\] 
\end{prop}

\begin{proof}
Let $\left\{(M_\alpha,\Sigma_\alpha,\Sigma'_\alpha):\alpha\in A\right\}$ be a non-empty collection of objects in $\mathcal{C}$. By (5.1), the coproduct, in $\mathcal{C}$, $\coprod_{\alpha\in A}(M_\alpha,\Sigma_\alpha,\Sigma'_\alpha)$ of the collection  $\left\{(M_\alpha,\Sigma_\alpha,\Sigma'_\alpha):\alpha\in A\right\}$ is given by 

\[(\bigoplus_{\alpha\in A}M_\alpha,\bigcup_{\alpha\in A}\iota_\alpha(\Sigma_\alpha),\bigcup_{\alpha\in A}\iota_\alpha(\Sigma'_\alpha))\]

\noindent together with cannonical morphisms $\iota_\alpha$, $\alpha\in A$. We prove that

\[\Gamma_I(\bigoplus_{\alpha\in A}M_\alpha,\bigcup_{\alpha\in A}\iota_\alpha(\Sigma_\alpha),\bigcup_{\alpha\in A}\iota_\alpha(\Sigma'_\alpha))\]

\noindent together with morphisms $\Gamma_I(\iota_\alpha)$, $\alpha\in A$ is a coproduct, in $\mathcal{G}$ of the collection $\left\{\Gamma_I(M_\alpha,\Sigma_\alpha,\Sigma'_\alpha):\alpha\in A\right\}$ of objects in $\mathcal{C}$. Let $(\Gamma,\Delta)$ be an object in $\mathcal{G}$ and let $\mu_\alpha:\Gamma_I(M_\alpha,\Sigma_\alpha,\Sigma'_\alpha)\rightarrow (\Gamma,\Delta)$ with $\alpha\in A$ be morphisms in $\mathcal{G}$. We define 

\[\nu:\Gamma_I(\bigoplus_{\alpha\in A}M_\alpha,\bigcup_{\alpha\in A}\iota_\alpha(\Sigma_\alpha),\bigcup_{\alpha\in A}\iota_\alpha(\Sigma'_\alpha))\rightarrow (\Gamma,\Delta)\]

\noindent as follows: Let $a\in\Sigma_\alpha$. We make $\nu(R\iota_\alpha(a))$ to be equal to $\mu_\alpha(Ra)$. Now, since the union $\bigcup_{\alpha\in A}\iota_\alpha(\Sigma_\alpha)$ is disjoint, this correspondence extends uniquely to a function

\[\nu:\left|\Gamma_I(\bigoplus_{\alpha\in A}M_\alpha,\bigcup_{\alpha\in A}\iota_\alpha(\Sigma_\alpha),\bigcup_{\alpha\in A}\iota_\alpha(\Sigma'_\alpha))\right|\rightarrow \left|(\Gamma,\Delta)\right|\]

\noindent We extend $\nu$ to a morphism in $G$. Clearly $\nu$, thus defined is a morphism in $\mathcal{G}$, $\nu$ satisfies the equation $\nu\Gamma_I(\iota_\alpha)=\mu_\alpha$ for all $\alpha\in A$, and it is unique with respect to this property. This concludes the proof.

\end{proof}

\noindent Note that, in general, functors of the form $\Gamma_I$ do not preserve direct products. This can easily be seen by the fact that while the graph

\[(\Gamma_{\mathbb{F}_3}(\mathbb{F}_3,\mathbb{F}_3\setminus\left\{0\right\},\emptyset))^2\]

\noindent has only one vertex, the graph

\[\Gamma_{\mathbb{F}_3}(\mathbb{F}_3^2,(\mathbb{F}_3\setminus\left\{0\right\})^2,\emptyset)\]

\noindent has two vertices, namely, $\mathbb{F}_3$-spaces generated by $(1,-1)$, and $(1,1)$ respectively.

\noindent Observe also that functors of the form $\Gamma_I$ need not, in general, preserve equalizers. To see this observe that while the equalizer of the pair of morphisms in $\mathcal{C}$

\[id_{\mathbb{F}_3}, -id_{\mathbb{F}_3}\]

\noindent as endomorphisms of $(\mathbb{F}_3,\mathbb{F}_3\setminus\left\{0\right\},\emptyset)$, is equal to the triple $(\left\{0\right\},\emptyset,\emptyset)$, the pair of morphisms

\[\Gamma_{\mathbb{F}_3}(id_{\mathbb{F}_3}),\Gamma_{\mathbb{F}_3}(-id_{\mathbb{F}_3}) \]

\noindent as endomorphisms, in $\mathcal{G}$, of $\Gamma_{\mathbb{F}_3}(\mathbb{F}_3,\mathbb{F}_3\setminus\left\{0\right\},\emptyset)$ is equal to the pair

\[id_{\Gamma_{\mathbb{F}_3}(\mathbb{F}_3,\mathbb{F}_3\setminus\left\{0\right\},\emptyset)}, id_{\Gamma_{\mathbb{F}_3}(\mathbb{F}_3,\mathbb{F}_3\setminus\left\{0\right\},\emptyset)}\]

\noindent whose equalizer, in $\mathcal{G}$, is equal to $\Gamma_{\mathbb{F}_3}(\mathbb{F}_3,\mathbb{F}_3\setminus\left\{0\right\},\emptyset)$

\section{Indecomposability in infinite dimensions}

\noindent In this section we use results of section 5 to extend and generalize the constructions and results of section 4. Our main results provide, under certain conditions, criteria for the existence of indecomposable modules of arbitrarily large and infinite ''rank''. We begin by precising on the notion of rank of a module.

\noindent Given a class of modules $\Omega$, we will say that a function $\rho:\Omega\rightarrow\mathbb{N}\cup\left\{\infty\right\}$ is a rank funcion if
\begin{enumerate}
\item $\rho(A\oplus B)\leq \rho(A)+\rho(B)$ for all $A,B\in\Omega$ such that $A\oplus B\in\Omega$.
\item $\rho(A)\leq \rho(B)$ for all $A,B\in\Omega$ such that $A\leq B$.
\end{enumerate}

\noindent Where we use the convention that $a+\infty =\infty+a=\infty$ for all $a\in\mathbb{N}\cup\left\{\infty\right\}$. Given a module $M\in\Omega$ we will call the number $\rho(M)$, the rank of $M$. We regard rank functions as measures of complexity on their domains of definition. More precisely, we regard the rank, $\rho(M)$, of a module $M$, with respect to a rank function $\rho$, as the number of dimensions of complexity of $M$, with respect to $\rho$. Functions $Gdim$, $dGdim$, and $\ell$, associating to each module $M$ its Goldie dimension, $GdimM$, its dual Goldie dimension, $dGdim$, and its length, $\ell(M)$, respectively, are all rank functions with domain $R$-Mod. Further, functions $Kdim$ and $dKdim$ associating to each module $M$ with Krull dimension, its Krull dimension, $KdimM$, and its dual Krull dimension, $dKdimM$, when these ordinals are finite, and $\infty$, when they are not, respectively, are both rank functions with domain equal to the class of all modules with Krull dimension.

We now set conditions on which we can guarantee the existence of indecomposable modules of arbitrarily large and infinite rank.

Given a direct system $\left\{M_\alpha,\varphi_{\alpha,\beta}:\alpha,\beta\in A\right\}$ in $R$-Mod, we will say that $\left\{M_\alpha,\varphi_{\alpha,\beta}:\alpha\in A\right\}$ is countable if the set $A$ is countable. Further, we will say that a class of modules $\Omega$ is complete, if $\Omega$ is closed under finite direct sums, direct summands, and direct limits of countable direct systems. $R$-Mod and the class of all modules with Krull dimension are both examples of complete classes. It is easily seen that given a class of modules $\Omega$, the class of all complete classes of modules containing $\Omega$, is closed under intersections. Thus, we can define the complete clousure, $\tilde{\Omega}$, of $\Omega$, as the intersection of all complete classes of modules containing $\Omega$. $\tilde{\Omega}$ thus defined, is complete, and is the minimal complete class containing $\Omega$. The class of all countably generated modules is a complete class, and further, it is equal to $\tilde{\Omega}$, where $\Omega$ is the class of all finitely generated modules.

Now, given an ideal $I$, and two modules $M$ and $N$, we will say that a morphism $f:M\rightarrow N$ is $I$-pure if $f(M)\leq_I N$. Further, we will say that a direct system $\left\{M_\alpha,\varphi_{\alpha,\beta}:\alpha\in A\right\}$ in $R$-Mod is $I$-pure if $\varphi_{\alpha,\beta}$ is an $I$-pure monomorphism for all $\alpha,\beta\in A$ such that $\alpha\leq\beta$. The following is an example of a countable $I$-pure direct system: Let $M_1\leq M_2\leq...$ be a chain of modules such that, for each $n\geq 1$, $M_n\leq_IM_{n+1}$. Let, for each $n,m$ such that $n\leq m$, $\varphi_{n,m}$, denote the inclusion of $M_n$ in $M_m$. Then, in this case, the system $\left\{M_n,\varphi_{n,m}:n,m\geq 1\right\}$ is a countable $I$-pure direct system. Recall that in section 4, for every field $k$, for every finite dimensional $k$-vector space $V$, and for any infinite admissible sequence $s$, a chain of $k\ltimes V$-modules, $M_{s_1}\leq M_{s_2}\leq...$ was defined. Recall also that, under these conditions, $M_{s_n}$ is always a $Soc(k\ltimes V)$-pure submodule of $M_{s_{n+1}}$ for all $n\geq 1$. Thus, in this case, for every infinite subsequence $t$ of $s$, the direct system $\left\{M_{t_n},\varphi_{n,m}:n,m\geq 1\right\}$ is a countable $Soc(k\ltimes V)$-pure direct system.

The following is the main result of this section. In it we provide criteria for the existence of indecomposable modules of arbitrarily large and infinite rank with respect to rank functions defined on complete classes of modules.

\begin{thm}
Let $\Omega$ be a class of modules and $\rho$ a rank function with domain $\Omega$. Suppose $\Omega$ is complete. If there exist, an ideal $I$ and an $I$-pure countable direct system $\left\{M_\alpha,\varphi_{\alpha,\beta}:\alpha,\beta\in A\right\}$ in $R$-Mod such that
\begin{enumerate}
\item $M_\alpha\in\Omega\cap\mathfrak{D}(I)$ for all $\alpha\in A$.
\item There exists a sequence $\alpha_1,\alpha_2,...$ in $A$ such that $\rho(M_{\alpha_i})<\rho(M_{\alpha_{i+1}})$ for all $i\geq 1$.
\item There exists $\kappa\in\mathbb{N}$ such that $cdim_IM_\alpha\leq\kappa$ for all but finetly many $\alpha\in A$.
\end{enumerate}
 
\noindent Then the following can be concluded

\begin{enumerate}
\item For each $m$ there exists an indecomposable $M\in\Omega$ such that $\rho(M)\geq m$
\item There exists an indecomposable $M\in\Omega$ such that $\rho(M)=\infty$
\end{enumerate}
\end{thm}

\noindent Observe that by the remarks preceeding (6.1), conjecture 4.1 implies the existence of indecomposable $k\ltimes V$-modules of infinite Goldie dimension. We regard (6.1) as an extension and generalization of (3.1)

\noindent We now reduce the proof of (6.1) to the proof of a series of lemmas.

\begin{lem}
Let $I$ be an ideal. Let $\left\{M_\alpha,\varphi_{\alpha,\beta}:\alpha,\beta\in A\right\}$ be a direct system in $R$-Mod. Suppose $\varphi_{\alpha,\beta}$ is a monomorphism for all $\alpha,\beta\in A$ such that $\alpha\leq\beta$. If $M_\alpha\in\mathfrak{D}(I)$ for all $\alpha\in A$, then $\varinjlim M_\alpha\in\mathfrak{D}(I)$.
\end{lem}

\begin{proof}
Let $\left\{M_\alpha,\varphi_{\alpha,\beta}:\alpha,\beta\in A\right\}$ be a direct system in $R$-Mod such that $\varphi_{\alpha,\beta}$ is a monomorphism for all $\alpha,\beta\in A$ such that $\alpha\leq\beta$. Suppose also that $M_\alpha\in\mathfrak{D}(I)$ for all $\alpha\in A$. From the fact that $I^2\varphi_{\alpha,\infty}(a)=\varphi_{\alpha,\infty}(I^2a)$ for all $a\in M_\alpha$ and for all $\alpha\in A$, and from the fact that $I^2M_\alpha=\left\{0\right\}$ for all $\alpha\in A$ it follows that $I^2\varinjlim M_\alpha=\left\{0\right\}$. Now, suppose that $a\in M_\alpha$ is such that $\varphi_{\alpha,\infty}(a)\in ann_I^*\varinjlim M_\alpha$. From the fact that $\varphi_{\alpha,\beta}$ is a monomorphism for all $\alpha,\beta\in A$ with $\alpha\leq\beta$ it follows that $\varphi_{\alpha,\infty}$ is also a monomorphism for all $\alpha\in A$. From this it follows that $a\in ann_I^*M_\alpha$, that is, $a\in IM_\alpha$. It follows that $\varphi_{\alpha,\infty}(a)\in I\varinjlim M_\alpha$. We conclude that $\varinjlim M_\alpha\in\mathfrak{D}(I)$. This concludes the proof.
\end{proof}

\begin{lem}
Let $\left\{\Gamma_\alpha,\varphi_{\alpha,\beta}:\alpha,\beta\in A\right\}$ be a direct system in $G$. Let $n\geq 1$. If the number of connected components of $\Gamma_\alpha$ is $\leq n$ for all $\alpha\in A$, then the number of connected components of $\varinjlim \Gamma_\alpha$ is $\leq n$.
\end{lem}

\begin{proof}
Let $\left\{\Gamma_\alpha,\varphi_{\alpha,\beta}:\alpha,\beta\in A\right\}$ be a direct system in $G$ such that $\Gamma_\alpha$ has at most $n$ connected components for each $\alpha\in A$. From the fact that $\varphi_{\alpha,\infty}$ is a morphism in $G$ for all $\alpha\in A$ it follows that the the subgraph $\varphi_{\alpha,\infty}(\Gamma_\alpha)$ of $\varinjlim \Gamma_\alpha$ has at most $n$ components for each $\alpha\in A$. Let $X_1,...,X_k$ be connected components of $\varinjlim\Gamma_\alpha$. Let $\alpha_1,...,\alpha_k\in A$ be such that 

\[X_i\cap \left|\varphi_{{\alpha_i},\infty}(\Gamma_{\alpha_i})\right|\neq\emptyset\]

\noindent for all $1\leq i\leq k$. Let $\gamma\in A$ be such that $\alpha_i\leq\gamma$ for all $1\leq i\leq k$. Then 

\[X_i\cap \left|\varphi_{\gamma,\infty}(\Gamma_{\gamma})\right|\neq\emptyset\]

\noindent It is easily seen that the sets $X_i\cap \left|\varphi_{\gamma,\infty}(\Gamma_{\gamma})\right|$ are all in different connected components of $\varphi_\gamma^\infty(\Gamma_\gamma)$. This concludes the proof. 
\end{proof}

\begin{lem}
Let $I$ be an ideal. Let $\left\{M_\alpha,\varphi_{\alpha,\beta}:\alpha,\beta\in A\right\}$ be an $I$-pure direct system in $R$-Mod. Let $n\geq 1$. Suppose $cdim_IM_\alpha$ exists and is $\leq n$ for all $\alpha\in A$, then $cdim_IM_\alpha$ exists and is $\leq n$.
\end{lem}

\begin{proof}
Let $\left\{M_\alpha,\varphi_{\alpha,\beta}:\alpha,\beta\in A\right\}$ be an $I$-pure direct system in $R$-Mod such that $cdim_IM_\alpha\leq n$ for all $\alpha\in A$. From the fact that $\varphi_{\alpha,\beta}$ is a monomorphism for all $\alpha,\beta\in A$ such that $\alpha\leq\beta$ it follows, from (6.2) that $cdim_I\varinjlim M_\alpha$ exists. Now, from the fact that $\varphi_{\alpha,\beta}(M_\alpha)\leq_IM_\beta$ for all $\alpha,\beta\in A$ with $\alpha\leq\beta$ it follows that for each $\alpha,\beta\in A$ such that $\alpha\leq\beta$, $\varphi_{\alpha,\beta}$ is a morphism in

\[\Hom_{\mathcal{C}}((M_\alpha,M_\alpha\setminus IM_\alpha,\emptyset),(M_\beta,M_\beta\setminus IM_\beta,\emptyset))\]

\noindent Thus 

\[\left\{(M_\alpha,M_\alpha\setminus IM_\alpha,\emptyset),\varphi_{\alpha,\beta}:\alpha,\beta\in A\right\}\]

\noindent is a direct system in $\mathcal{C}$. The direct limit, in $\mathcal{C}$ of this direct system is 

\[(\varinjlim M_\alpha,\varinjlim M_\alpha\setminus I\varinjlim M_\alpha,\emptyset)\]

\noindent which easily followa from the equations

\[I\bigcup_{\alpha\in A}\varphi_{\alpha,\infty}(M_\alpha)=\bigcup_{\alpha\in A}I\varphi_{\alpha,\infty}(M_\alpha)\]

\noindent From this and from (6.3) it follows that $\Gamma_I(\varinjlim M_\alpha)$ has at most $n$ connected components.
\end{proof}

\

\noindent\textit{\textbf{Proof of 6.1}} Let $\left\{M_\alpha,\varphi_{\alpha,\beta}:\alpha,\beta\in A\right\}$ be an $I$-pure direct system in $R$-Mod satisfying conditions 1,2, and 3 of (6.1). The proof of conclusion 1 in (6.1) is completely analogous to the proof of (4.3) substituting only the function $Gdim$ by $\rho$. We prove conclusion 2. From the fact that $\Omega$ is a complete class it follows that $\varinjlim M_\alpha\in\mathfrak{D}(I)$. From this, from condition 2 and from (6.3) it follows that $cdim_I\varinjlim M_\alpha\leq \kappa$. From this and from (2.3) it follows that $\varinjlim M_\alpha$ admits finite indecomposable direct sum decompositions, and that every direct sum decomposition of $\varinjlim M_\alpha$ has cardinality at most $\kappa$. Let $\varinjlim M_\alpha=\bigoplus_{i=1}^nN_i$ be a finite indecomposable direct sum decomposition of $\varinjlim M_\alpha$. From condition 3 it follows that $\rho(\varinjlim M_\alpha)=\infty$. The left hand side of this identity is at most $\sum_{i=1}^n\rho(N_i)$. It follows that there exists $1\leq t\leq n$ such that $\rho(N_t)=\infty$. $N_t\in\Omega$ since $\Omega$ is complete. Thus $M=N_t$ is indecomposable, $M\in\Omega$, and $\rho(M)=\infty$. This concludes the proof. $\square$

\

\noindent We dedicate the rest of this section to present a version of (6.1) in which combinatorial dimensions are substituted by fundamental combinatorial dimensions. We first establish certain terminology.

We will say that a direct system $\left\{(M_\alpha,\Sigma_\alpha,\Sigma'_\alpha),\varphi_{\alpha,\beta}:\alpha,\beta\in A\right\}$ in $\mathcal{C}$ is fundamental with respect to an ideal $I$ if the following conditions hold

\begin{enumerate}
\item The system $\left\{M_\alpha,\varphi_{\alpha,\beta}:\alpha,\beta\in A\right\}$ is $I$-pure
\item $\Sigma_\alpha=M_\alpha\setminus IM_\alpha$ for all $\alpha\in A$.
\item $\Sigma'_\alpha$ is fundamental in $M_\alpha$, with respect to $I$, for all $\alpha\in A$.
\item For each $\varphi_{\alpha,\infty}(a)\in\bigcup_{\alpha\in A}\varphi_{\alpha,\infty}(M_\alpha)$ there exist $\beta\in A$ and $u\in M_\beta$ such that, for each $\gamma\in A$ with $\gamma\geq \alpha,\beta$, we have that

\[\varphi_{\beta,\gamma}(u)\notin\sum_{\left|\Gamma_I(M_\gamma)\right|\setminus C_{\varphi_{\alpha,\gamma}(a)}^I}Rb\]
\end{enumerate}

\noindent We regard the following lemma as a version, for fundamental combinatorial dimensions, of (6.2).

\begin{lem}
Let $I$ be an ideal. Let $\left\{(M_\alpha,\Sigma_\alpha,\Sigma'_\alpha),\varphi_{\alpha,\beta}:\alpha,\beta\in A\right\}$ be a direct system in $\mathcal{C}$. If the system $\left\{(M_\alpha,\Sigma_\alpha,\Sigma'_\alpha),\varphi_{\alpha,\beta}:\alpha,\beta\in A\right\}$ is fundamental with respect to $I$, then 

\[\varinjlim M_\alpha\in\mathfrak{F}(I)\]

\noindent Moreover, in this case, the set $\bigcup_{\alpha\in A}\varphi_\alpha^\infty(\Sigma'_\alpha)$ is fundamental in $\varinjlim M_\alpha$ with respect to $I$.
\end{lem}

\begin{proof}
Let $\left\{(M_\alpha,\Sigma_\alpha,\Sigma'_\alpha),\varphi_{\alpha,\beta}:\alpha,\beta\in A\right\}$ be a direct system in $\mathcal{C}$, fundamental with respect to $I$. From the fact that $\varphi_{\alpha,\beta}$ is a monomorphism for all $\alpha,\beta\in A$ such that $\alpha\leq \beta$, from (6.2) and from the fact that $M_\alpha\in\mathfrak{F}(I)$ for all $\alpha\in A$ it follows that $\varinjlim M_\alpha\in\mathfrak{D}(I)$. We prove now that $\bigcup_{\alpha\in A}\varphi_{\alpha,\infty}(\Sigma'_\alpha)$ is fundamental in $\varinjlim M_\alpha$ with respect to $I$. From the fact that $\Sigma'_\alpha$ generates $M_\alpha$ for all $\alpha\in A$ and from the equation

\[\varinjlim M_\alpha=\bigcup_{\alpha\in A}\varphi_{\alpha,\infty}(M_\alpha)\]

\noindent it follows that $\bigcup_{\alpha\in A}\varphi_{\alpha,\infty}(\Sigma'_\alpha)$ generates $\varinjlim M_\alpha$. Thus, it is enough to prove, for each $\varphi_{\alpha,\infty}(a)\in\bigcup_{\alpha\in A}\varphi_{\alpha,\infty}(M_\alpha)$, that

\[\sum_{\Gamma_I(\varinjlim M_\alpha)\setminus C_{\varphi_{\alpha,\infty}(a)}^I}Rb\neq \varinjlim M_\alpha\]

\noindent This follows from the fact that $\left\{(M_\alpha,\Sigma_\alpha,\Sigma'_\alpha),\varphi_{\alpha,\beta}:\alpha,\beta\in A\right\}$ is fundamental with respect to $I$, together with equations

\[C_{\varphi_{\alpha,\infty}(a)}^I=\Gamma_I(\varphi_{\alpha,\infty})(C_a^I)\]

\noindent which easily follows from (5.5). This concludes the proof.
\end{proof}

\noindent We again regard the following lemma as a version, for fundamental combinatorial dimensions, of (6.3).

\begin{lem}
Let $I$ be an ideal. Let $n\geq 1$. Let $\left\{M_\alpha,\varphi_{\alpha,\beta}:\alpha,\beta\in A\right\}$ be a direct system in $R$-Mod such that 

\[\varphi_{\alpha,\beta}(M_\alpha\setminus IM_\alpha)\subseteq M_\beta\setminus IM_\beta\]

\noindent for all $\alpha,\beta\in A$ with $\alpha\leq\beta$. Suppose there exists a collection of sets $\left\{\Sigma_\alpha :\alpha\in A\right\}$ satisfying the following conditions.

\begin{enumerate}
\item $\Sigma_\alpha\subseteq M_\alpha\setminus IM_\alpha$ for all $\alpha\in A$.
\item $\varphi_{\alpha,\beta}(\Sigma_\alpha)\subseteq \Sigma_\beta$ for all $\alpha,\beta\in A$ such that $\alpha\leq\beta$.
\item The system $\left\{(M_\alpha,M_\alpha\setminus IM_\alpha, \Sigma_\alpha),\varphi_{\alpha,\beta}:\alpha,\beta\in A\right\}$ is fundamental with respect to $I$
\end{enumerate}

\noindent If $fcdim_IM_\alpha\leq n$ for all $\alpha\in A$, then $fcdim_I\varinjlim M_\alpha\leq n$.
\end{lem}

\begin{proof}
Let $\left\{M_\alpha,\varphi_{\alpha,\beta}:\alpha,\beta\in A\right\}$ be a direct system in $R$-Mod such that 

\[\varphi_{\alpha,\beta}(M_\alpha\setminus IM_\alpha)\subseteq M_\beta\setminus IM_\beta\]

\noindent for all $\alpha,\beta\in A$ with $\alpha\leq\beta$. Let $\left\{\Sigma_\alpha:\alpha\in A\right\}$ be a collection of sets satisfying conditions 1, 2, and 3 above. Suppose $fcdim_IM_\alpha\leq n$ for all $\alpha\in A$. From (6.5) it follows that 

\[\varinjlim M_\alpha\in\mathfrak{F}(I)\]

\noindent and that $\bigcup_{\alpha\in A}\varphi_\alpha^\infty(\Sigma_\alpha)$ is fundamental in $\varinjlim M_\alpha$ with respect to $I$. Now, for each $\alpha\in A$, let $\Gamma_\alpha$ be the subgraph of $\Gamma_I(M_\alpha)$ generated by $\bigcup_{a\in\Sigma_\alpha}C_a^I$. Since $\Gamma_I(\varphi_{\alpha,\infty})$ is a morphism in $G$ and $\varphi_{\alpha,\beta}(\Sigma_\alpha)\subseteq \Sigma_\beta$ for all $\alpha,\beta\in A$ such that $\alpha\leq\beta$, it follows that $\left\{\Gamma_\alpha,\Gamma_I(\varphi_{\alpha,\beta}):\alpha,\beta\in A\right\}$ is a direct system in $G$. Now, since for each $\alpha\in A$, $fcdim_IM_\alpha$ is the number of connected components of $\Gamma_\alpha$, it follows, by (6.3), that $\varinjlim \Gamma_\alpha$ has at most $n$ connected components. The result follows from this and from the fact that $\varinjlim \Gamma_\alpha$ is the subgraph of $\Gamma_I(\varinjlim M_\alpha)$ generated by $\bigcup_{\alpha\in A}\varphi_{\alpha,\infty}(\Sigma_\alpha)$. 
\end{proof}

\noindent The next result is a version of (6.1) in which combinatorial dimensions are substituted by fundamental combinatorial dimensions.

\begin{thm}
Let $\Omega$ be a class of modules. Let $\rho$ be a rank function with domain $\Omega$. If $\Omega$ is complete, and there exist, an ideal $I$, and a countable $I$-pure direct system $\left\{M_\alpha,\varphi_{\alpha,\beta}:\alpha,\beta\in A\right\}$ in $R$-Mod, such that
\begin{enumerate}
\item $\left\{M_\alpha,\varphi_{\alpha,\beta}:\alpha,\beta\in A\right\}$ satisfies the conditions of (6.6).
\item There exists a sequence $\alpha_1,\alpha_2,...$ in $A$ such that $\rho(M_{\alpha_i})<\rho(M_{\alpha_{i+1}})$ for all $i\geq 1$ 
\item There exists $\phi\geq 1$ such that $fcdim_IM_\alpha\leq \phi$ for all $\alpha\in A$
\end{enumerate}
Then the following can be concluded
\begin{enumerate}
\item For each $n\geq 1$ there exists an indecomposable $M\in\Omega$ such that $\rho(M)\geq n$.
\item There exists an indecomposable $M\in\Omega$ such that $\rho(M)=\infty$
\end{enumerate}
\end{thm}

\begin{proof}
The proof is completely analogous to the proof of (6.1) subsituting only (6.3), and (6.4) for (6.5), and (6.6). $\square$
\end{proof}

\noindent Observe that direct systems defined in section 4 and in the remarks preceeding (6.1) satisfy condition 1 of (6.6). Moreover, if we assume that infinite sequences $s$ and $t$ satisfy conjecture 4.2 with respect to finite dimensional $k$-vector space $V$ and integer $\phi_V$, then the system $\left\{M_{t_n},\varphi_{n,m}:n,m\geq 1\right\}$ satisfies conditions 2 and 3 of (6.7). We do not assume however the system $\left\{M_{t_n},\varphi_{n,m}:n,m\geq 1\right\}$ to satisfy conditions 2 and 3 of (6.6) with respect to any collection $\left\{\Sigma_{t_n}:n\geq 1\right\}$. We conjecture the following.

\begin{conj}
Let $k$ be a field. Let $V$ be a finite dimensional $k$-vector space. If $dim_kV$ is sufficently large, then there exists an integer $\phi_V$, an infinite admissible sequence $s$, an infinite subsequence $t$ of $s$ and a sequence of sets $\left\{\Sigma_{t_n}:n\geq 1\right\}$ such that
\begin{enumerate}
\item $M_{t_n}\in\mathfrak{F}(Soc(k\ltimes V))$ for all $n\geq 1$.
\item $fcdimM_{t_n}\leq \phi_V$ for all $n$
\item The system $\left\{M_{t_n},\varphi_{n,m}:n,m\geq 1\right\}$ satisfies conditions 2 and 3 of (6.6) with respect to collection $\left\{\Sigma_{t_n}:n\geq 1\right\}$
\end{enumerate}
\end{conj}

\noindent Observe that (6.7), together with a positive answer to (6.8), implies, for finite dimensional $k$-vector space $V$ with $dim_kV$ is sufficently large, the existence of indecomposable $k\ltimes V$-modules of arbitrarily large and infinite Goldie dimension.

\section{Two final computations}

\noindent In this final section we present two approximation results on the computation of graphs of cyclic modules. The results presented in this section can be regarded as being of the same type as (5.3), (5.4), and (5.5). (5.3), (5.4), and (5.5) can be regarded as approximation results in the calculus of graphs of cyclic modules in the following way: (5.3) says that if a pair $(M,\Sigma)$ can be approximated as the union $(\bigcup_{\alpha\in A}M_\alpha,\bigcup_{\alpha\in A}\Sigma_\alpha)$ of a directed collection of pairs $\left\{(M_\alpha,\Sigma_\alpha):\alpha\in A\right\}$, then the computation of the $I$-graph of cyclic modules, $\Gamma_I(M,\Sigma)$, of $(M,\Sigma)$ with respect to any ideal $I$ can be reduced to the computation of each of the $I$-graphs of cyclic modules, $\Gamma_I(M_\alpha,\Sigma_\alpha)$, of $(M_\alpha,\Sigma_\alpha)$, $\alpha\in A$, together with the computation of their union. Moreover, (5.4) and (5.5) say that if an object $(M,\Sigma,\Sigma')$ in $\mathcal{C}$ can be approximated as the direct limit of a direct system $\left\{(M_\alpha,\Sigma_\alpha,\Sigma'_\alpha),\varphi_{\alpha,\beta}:\alpha,\beta\in A\right\}$ in $\mathcal{C}$, or as the coproduct of the non-empty collection $\left\{(M_\alpha,\Sigma_\alpha,\Sigma'_\alpha):\alpha\in A\right\}$ of objects in $\mathcal{C}$ respectively, then the computation of the $I$-graph of cyclic modules, $\Gamma_I(M,\Sigma,\Sigma')$ of $(M,\Sigma,\Sigma')$ with respect to any ideal $I$, can be reduced to the computation of each of the $I$-graphs of cyclic modules $\Gamma_I(M_\alpha,\Sigma_\alpha,\Sigma'_\alpha)$, together with the computation of their direct limit in the first case and together with the computation of their coproduct in the second case. Observe that the first of these two cases implies that the computation of the $I$-graph of cyclic modules, $\Gamma_I(M,\Sigma,\Sigma')$, of any object $(M,\Sigma,\Sigma')$ can be reduced to the computation of $I$-graphs of cyclic modules of objects of the form $(M_\alpha,\Sigma_\alpha,\Sigma'_\alpha)$, where $M_\alpha$ is finitely generated and $\Sigma_\alpha$ and $\Sigma'_\alpha$ are finite for all $\alpha\in A$. 

We regard the following lemma as a version of (5.3) or (5.4) in which approximations are performed on ideals rather than on objects in $\mathcal{C}$.

\begin{lem}
Let $(M,\Sigma,\Sigma')$ be an object in $\mathcal{C}$. Let $\left\{I_\alpha:\alpha\in A\right\}$ be a collection of ideals. If the set $\left\{I_\alpha:\alpha\in A\right\}$ is directed with respect to the order given by inclusion, then

\[\Gamma_{\sum_{\alpha\in A}I_\alpha}(M,\Sigma,\Sigma')=\Gamma_{\bigcup_{\alpha\in A}I_\alpha}(M,\Sigma,\Sigma')=\bigcup_{\alpha\in A}\Gamma_{I_\alpha}(M,\Sigma,\Sigma')\]
\end{lem}

\begin{proof}
Let $\left\{I_\alpha:\alpha\in A\right\}$ be a collection of ideals, directed with respect to the order given by inclusion. In this case the equality 

\[\bigcup_{\alpha\in A}I_\alpha=\sum_{\alpha\in A}I_\alpha\]

\noindent holds. The left part of the equality in the statement of the lemma follows. Now, clearly the identity

\[\left|\Gamma_{\sum_{\alpha\in A}I_\alpha}(M,\Sigma,\Sigma')\right|=\left|\bigcup_{\alpha\in A}\Gamma_{I_\alpha}(M,\Sigma,\Sigma')\right|\]

\noindent holds. Thus, we only need to prove that the corresponding adjacency relations are equal. The adjacency relation of $\bigcup_{\alpha\in A}\Gamma_{I_\alpha}(M,\Sigma,\Sigma')$ is clearly contained in the adjacency relation of $\Gamma_{\sum_{\alpha\in A}I_\alpha}(M,\Sigma,\Sigma')$. Now, let $a,b\in\Sigma$. Suppose $Ra$ and $Rb$ are adjacent in $\Gamma_{\sum_{\alpha\in A}I_\alpha}(M,\Sigma,\Sigma')$, then there exist $\alpha_1,...,\alpha_n$ such that 

\[(\sum_{i=1}^nI_i)a\cap(\sum_{i=1}^nI_i)b\neq\left\{0\right\}\]

\noindent Let $\beta\in A$ be such that $\alpha_i\leq\beta$ for all $1\leq i\leq n$. Then

\[I_\beta a\cap I_\beta b\neq\left\{0\right\}\]

\noindent that is, $Ra$ and $Rb$ are adjacent in $\bigcup_{\alpha\in A}\Gamma_{I_\alpha}(M,\Sigma,\Sigma')$. This concludes the proof.
\end{proof}

\noindent Observe that, as in (5.4), (7.1) implies that, the computation of the $I$-graph of cyclic modules, $\Gamma_I(M,\Sigma,\Sigma')$, of any object $(M,\Sigma,\Sigma')$ in $\mathcal{C}$, with respect to any ideal $I$, can be reduced to the computation of graphs of cycluc modules of the form $\Gamma_{I_\alpha}(M,\Sigma,\Sigma')$, $\alpha\in A$, where $I_\alpha$ is a finitely generated ideal contained in $I$ for every $\alpha\in A$, together with the computation of their union.

The next lemma can be regarded as a version, in the case where the base ring $R$ is a commutative domain and the module $M$ is torsion free, of (7.1), in which approximations are performed by products and intersections of finite collections of ideals with no assumptions on their order structure rather than by sums and unions of collections of ideals, directed with respect to inclusion. First recall that the intersection $\bigcap_{\alpha\in A}\Gamma_\alpha$, of a family of graphs $\left\{\Gamma_\alpha:\alpha\in A\right\}$ is defined as follows:

\begin{enumerate}
\item $\left|\bigcap_{\alpha\in A}\Gamma_\alpha\right|=\bigcap_{\alpha\in A}\left|\Gamma_\alpha\right|$
\item Given $a,b\in \bigcap_{\alpha\in A}\left|\Gamma_\alpha\right|$, we say that $a$ and $b$ are adjacent in $\bigcap_{\alpha\in A}\Gamma_\alpha$ if $a$ and $b$ are adjacent in $\Gamma_\alpha$ for all $\alpha\in A$.
\end{enumerate}

\noindent Observe that, for each $\alpha\in A$, the graph $\bigcap_{\alpha\in A}\Gamma_\alpha$ is contained, as a subgraph, in $\Gamma_\alpha$, and that $\bigcap_{\alpha\in A}\Gamma_\alpha$ is maximal with respect to this property.

\begin{lem}
Suppose the ring $R$ is a commutative domain. Let $(M,\Sigma,\Sigma')$ be an object in $\mathcal{C}$. Let $\left\{I_1,...,I_n\right\}$ be a finite set of ideals. If $M$ is torsion free, then

\[\Gamma_{\prod_{i=1}^nI_i}(M,\Sigma,\Sigma')=\Gamma_{\bigcap_{i=1}^n}(M,\Sigma,\Sigma')=\bigcap_{i=1}^n\Gamma_{I_i}(M,\Sigma,\Sigma')\]
\end{lem}

\begin{proof}
Suppose $R$ is a commutative domain. Suppose also that $M$ is torsion-free. Again, it is easily seen that the identity

\[\left|\Gamma_{\prod_{i=1}^nI_i}(M,\Sigma,\Sigma')\right|=\left|\Gamma_{\bigcap_{i=1}^nI_i}(M,\Sigma,\Sigma')\right|=\left|\bigcap_{i=1}^n\Gamma_{I_i}(M,\Sigma,\Sigma')\right|\]

\noindent holds. Thus, again, we only need to prove that the corresponding adjacency relations are equal. Let $a,b\in\Sigma$. It is easily seen that if $Ra$ and $Rb$ are adjacent in $\Gamma_{\prod_{i=1}^nI_i}(M,\Sigma,\Sigma')$, then $Ra$ and $Rb$ are adjacent in $\Gamma_{\bigcap_{i=1}^nI_i}(M,\Sigma,\Sigma')$, and that if $Ra$ and $Rb$ are adjacent in $\Gamma_{\bigcap_{i=1}^nI_i}(M,\Sigma,\Sigma')$, then $Ra$ and $Rb$ are adjacent in $\bigcap_{i=1}^n\Gamma_{I_i}(M,\Sigma,\Sigma')$. Now, suppose $Ra$ and $Rb$ are adjacent in $\Gamma_{I_i}(M,\Sigma,\Sigma')$ for all $1\leq i\leq n$. Let $r_i\in I_i\setminus\left\{0\right\}$ be such that $r_ia=r_ib$ for all $1\leq i\leq n$. Let $r$ be equal to $\prod_{i=1}^nr_i$. Then $r\in\prod_{i=1}^nI_i\setminus\left\{0\right\}$ and $ra=rb$. We conclude that $Ra$ and $Rb$ are adjacent in $\Gamma_{\prod_{i=1}^nI_i}(M,\Sigma,\Sigma')$. This concludes the proof.  
\end{proof}

\noindent Thus, in the case where the base ring $R$ is a commutative domain and the ideal $I$ can be approximated as the intesection of a finite family of ideals $\left\{I_1,...,I_n\right\}$, then the computation of the $I$-graph of cyclic modules, $\Gamma_I(M,\Sigma,\Sigma')$, of any object $(M,\Sigma,\Sigma')$ in $\mathcal{C}$ such that $M$ is torsion-free can be reduced to the computation of each of the graphs $\Gamma_{I_i}(M,\Sigma,\Sigma')$, with $1\leq i\leq n$, together with the computation of their intersection. In particular, when the base ring $R$ is a Lasker domain (e.g. $R$ is noetherian), the computation of the $I$-graph of cyclic modules $\Gamma_I(M,\Sigma,\Sigma')$, of any object $(M,\Sigma,\Sigma')$ such that $M$ is torsion free, with respect to any ideal $I$, reduces to the computation of a finite number of graphs of cyclic modules of $(M,\Sigma,\Sigma')$ with respect to primary ideals, together with the computation of their intersection.

\section{Bibliography}

\noindent [1] J. Adamek, H. Herrlich, G. E. Strecker, \textit{Abstract and Concrete Categories. The Joy of Cats}, John Wiley and Sons, (1990).

\

\noindent [2] F. Anderson, K. Fuller, \textit{Rings and Categories of Modules}, Springer-Verlag, (1991).

\

\noindent [3] I. Assem, J. A. de la Peña, \textit{On the Tameness of Trivial Extension Algebras}, Fundamenta Mathematicae, \textbf{149}, (1996), 171-181.

\

\noindent [4] G. Azuaya, \textit{Corrections and Supplementaries to my Paper Concerning Remak-Krull-Schmidt's Theorem}, Nagoya Math. J. \textbf{1}, (1950), 117-124.

\

\noindent [5] B. Bollobas, \textit{Modern Graph Theory}, Springer (1998).

\

\noindent [6] J. A. Bundy, U. S. R Murty, \textit{Graph Theory With Applications}, North-Holland, (1976).

\

\noindent [7] J. A. Bundy, U. S. R Murty, \textit{Graph Theory}, Springer-Verlag, (2008).

\

\noindent [8] R. Diestel, \textit{Graph Theory, Third Edition}, Springer-Verlag, (2005).

\

\noindent [9] D. S. Dummit, R. M. Foote, \textit{Abstract Algebra}, John Wiley and Sons, (2004).

\

\noindent [10] A. Facchini, \textit{Module Theory, Endomorphism Rings and Direct Sum Decompositions in Some Classes of Modules}, Birkhäuser-Verlag, (1998). 

\

\noindent [11] A. Facchini, D. Herbera, \textit{Modules With Only Finitely Many Direct Sum Decompositions up to Isomorphisms}, Irish Math. Soc. Bull. \textbf{50} (2003), 51-69.

\

\noindent [12] R. M. Fossum, P. A. Griffith, I. Reiten, \textit{Trivial Extensions of Abelian Categories}, Lecture Notes in Math. \textbf{456}, Springer. 

\

\noindent [13] L. Fuchs, \textit{Infinite Abelian Groups} Vol. II, Elsevier (1973). 

\

\noindent [14] L.Fuchs, \textit{The Existence of Indecomposable Abelian Groups of Arbitrary Power}, Acta Math. Acad. Sci. Hungar. \textbf{10} (1959), 453-457.

\

\noindent [15] L. Fuchs, \textit{On a Directly Indecomposable Abelian Group of Power Greater than Continuum}, Acta. Math. Acad. Sci. Hungar. \textbf{8} (1957), 453-454.

\

\noindent [16] S. Kabbaj, N. Mahdou, \textit{Trivial Extensions of Local Rings and a Conjecture of Costa}, Lecture Notes in Pure and Applied Mathematics-Dekker, \textbf{231}, (2002), 301-311.

\

\noindent [17] Y. Kitamura, \textit{On Quotient Rings of Trivial Extension Algebras}, Proc. Amer. Math. Soc. \textbf{88}, No. 3, (1983), 391-396.

\

\noindent [18] S. MacLane, \textit{Categories for the Working Mathematician}, Springer Verlag, (1997).

\

\noindent [19] T. Y. Lam, \textit{Lectures on Modules and Rings}, Springer, (1998).

\

\noindent [20] L. H. \textit{Rowen Ring Theory, Student Edition}, Academic Press, (1991).

\

\noindent [21] P. Vamos, \textit{The Holy Grail of Algebra: Seeking Complete Sets of Invariants, Abelian Goups and Modules}, A. Facchini, C. Menini, eds. Math. and its Appl. 343, Kluwer Drordrecht, 1995, 475-483.

\

\noindent [22] R. Wisbauer, \textit{Foundations of Module and Ring Theory}, Gordon and Reach (1991).

\end{document}